\documentclass{amsart}

\usepackage{amsxtra,amssymb,amsthm,amsmath,amscd,url,xypic}
\usepackage[utf8]{inputenc}
\usepackage{dsfont}

\renewcommand{\leq}{\leqslant}
\renewcommand{\geq}{\geqslant}

\numberwithin{equation}{section}
\def\stacksum#1#2{{\stackrel{{\scriptstyle #1}}
{{\scriptstyle #2}}}}

\newcommand{\Cc}{\mathbf{C}}

\newcommand{\Zz}{\mathbf{Z}}

\newcommand{\Rr}{\mathbf{R}}

\newcommand{\Qq}{\mathbf{Q}}

\newcommand{\proba}{\text{\boldmath$P$}}
\newcommand{\expect}{\text{\boldmath$E$}}

\newcommand{\charfun}{\mathds{1}}

\newcommand{\mods}[1]{\,(\mathrm{mod}\,{#1})}

\newcommand{\ra}{\rightarrow}

\newcommand{\sing}[1]{\mathfrak{S}(\uple{{#1}})}
\newcommand{\singn}[1]{\mathfrak{S}({{#1}})}
\newcommand{\uple}[1]{\text{\boldmath${#1}$}}
\newcommand{\compos}[2]{\uple{{#1}}\odot \uple{{#2}}}
\newcommand{\composn}[2]{{{#1}}\odot {{#2}}}

\newcommand{\stirling}[2]{\genfrac{\{}{\}}{0pt}{}{{#1}}{{#2}}}

\DeclareMathOperator{\res}{Res}

\DeclareMathOperator{\Reel}{Re}

\DeclareMathOperator{\peg}{peg}

\newcommand{\eps}{\varepsilon}

\DeclareMathSymbol{\gena}{\mathord}{letters}{"3C}
\DeclareMathSymbol{\genb}{\mathord}{letters}{"3E}

\def\sumb{\mathop{\sum\Bigl.^{\flat}}\limits}

\def\multsum{\mathop{\sum\cdots\sum}\limits}
\def\dblsum{\mathop{\sum \sum}\limits}

\def\sums{\mathop{\sum \Bigl.^{*}}\limits}

\theoremstyle{plain}
\newtheorem{theorem}{Theorem}[section]

\newtheorem{lemma}[theorem]{Lemma}
\newtheorem{corollary}[theorem]{Corollary}

\newtheorem{proposition}[theorem]{Proposition}

\theoremstyle{remark}
\newtheorem{remark}[theorem]{Remark}

\theoremstyle{definition}

\newtheorem{example}[theorem]{Example}

\begin{document}

\title[Averages of Euler products]{Averages of Euler products,
  distribution of singular series and the ubiquity of Poisson
  distribution}
 
\author{Emmanuel  Kowalski}
\address{ETH Z\"urich -- D-MATH\\
  R\"amistrasse 101\\
  8092 Z\"urich\\
  Switzerland} 
\email{kowalski@math.ethz.ch}

\date{}

\subjclass[2010]{11P32, 11N37, 11K65} \keywords{Singular series, prime
  $k$-tuples conjecture, Bateman-Horn conjecture, limiting
  distribution, Euler product, moments, Poisson distribution}

\begin{abstract}
  We discuss in some detail the general problem of computing averages
  of convergent Euler products, and apply this to examples arising
  from singular series for the $k$-tuple conjecture and more general
  problems of polynomial representation of primes.  We show that the
  ``singular series'' for the $k$-tuple conjecture have a limiting
  distribution when taken over $k$-tuples with (distinct) entries of
  growing size.
  We also give conditional arguments that would imply that the number
  of twin primes (or more general polynomial prime patterns) in
  suitable short intervals are asymptotically Poisson distributed.
\end{abstract}

\maketitle

\section{Introduction}

Euler products over primes are ubiquitous in analytic number theory,
going back to Euler's proof that there are infinitely many prime
numbers based on the behavior of the zeta function $\zeta(s)$ as $s\ra
1$. As defining $L$-functions of various types, Euler products are
particularly important, and their properties remain very mysterious.
In this paper, we consider the issue of the average or statistical
behavior of another important class of Euler products, the so-called
\emph{singular series}, arising in counting problems for certain
``patterns'' of primes (singular series also occur in many problems of
additive number theory or diophantine geometry, but we do not consider
these here).
\par
The first type of prime patterns are the prime $k$-tuples, which are
the subject of a famous conjecture of Hardy and Littlewood. Let $k\geq
1$ be an integer and let $\uple{h}=(h_1,\ldots,h_k)$ be a $k$-tuple of
integers with $h_i\geq 1$ for all $i$. Let then
$$
\pi(N;\uple{h})=|\{
n\leq N\,\mid\, n+h_i\text{ is prime for $1\leq i\leq k$}
\}|
$$
be the counting function for primes represented by this $k$-tuple;
note that, for instance, $\uple{h}=(1,3)$ leads to the function
counting twin primes up to $N$.
\par
For any prime number $p$, let $\nu_p(\uple{h})$ denote the cardinality
of the set
$$
\{h_1,\ldots, h_k\}\mods {p}
$$
of the reductions of the $h_i$ modulo $p$. Note that $1\leq
\nu_p(\uple{h})\leq \min(k,p)$ for all $p$, and that if we assume (as
we now do) that the $h_i$'s are distinct, then $\nu_p(\uple{h})=k$ for
all sufficiently large $p$.
\par
The \emph{singular series} associated with $\uple{h}$ is defined as
the Euler product
\begin{equation}\label{eq-singular}
\sing{h}=\prod_{p}{
\Bigl(1-\frac{\nu_p(\uple{h})}{p}\Bigr)
\Bigl(1-\frac{1}{p}\Bigr)^{-k}}
=
\prod_{p}{\Bigl(1-\frac{\nu_p(\uple{h})-1}{p-1}\Bigr)
\Bigl(1-\frac{1}{p}\Bigr)^{1-k}}
\end{equation}
which is absolutely convergent (as will be checked again later; here
and throughout the paper, as usual, $p$ is restricted to prime
numbers).
\par
The significance of this value is found in the Hardy-Littlewood prime
$k$-tuple conjecture (originally stated in~\cite{hl}), which states
that we should have
\begin{equation}\label{eq-ktuple}
  \pi(N;\uple{h})= \sing{h}\frac{N}{(\log N)^k}(1+o(1)),
\quad\quad \text{ as } N\ra +\infty,
\end{equation}
and in particular, if $\sing{h}\not=0$, there should be infinitely
many integers $n$ such that $n+h_1$, \ldots, $n+h_k$ are
simultaneously prime. Of course, if $k\geq 2$, this is still
completely open, but let us mention that from sieve methods, it
follows that
$$
\pi(N;\uple{h})
\leq 2^kk!(1+o(1))\sing{h}\frac{N}{(\log N)^k}
$$
as $N\ra +\infty$ (see, e.g.,~\cite[Th. 6.7]{ik} or~\cite[Ch. 4,
Th. 5.3]{halberstam-richert}), showing that the singular series does
arise naturally. Also some other previously inaccessible additive
problems with primes, related to counting arithmetic progressions (of
fixed length) of primes are currently being attacked with striking
success by B. Green and T. Tao (see~\cite{green-tao}).
\par
More generally, one considers polynomial prime patterns. First, a
finite family $\uple{f}=(f_1,\ldots, f_m)$ of polynomials in $\Zz[X]$
of degrees $\deg(f_j)\geq 1$ is said to be \emph{primitive} if the
$f_j$ are distinct, and each $f_j$ is irreducible, has positive
leading coefficient, and the gcd of its coefficients is $1$.
\par
If $\uple{f}$ is primitive, we say that an integer $n\geq 1$ is an
$\uple{f}$-prime seed if $f_1(n)$, \ldots, $f_m(n)$ are all (positive)
primes. Then we denote by
$$
\pi(N;\uple{f})=|\{n\leq N\,\mid\, n\text{ is an $\uple{f}$-prime
  seed}\}|
$$
for $N\geq 1$ the counting function for those prime seeds. Moreover,
let
$$
\peg(\uple{f})=\prod_{j=1}^m{\deg(f_j)}.
$$
\par
A generalization of the $k$-tuple conjecture, due to Bateman and
Horn~\cite{bateman-horn},\footnote{\ The qualitative version of which
  is due to Schinzel~\cite{schinzel}.}  states that
\begin{equation}\label{eq-bh}
\pi(N;\uple{f})
\sim \frac{1}{\peg(\uple{f})}\sing{f}\frac{N}{(\log N)^{m}},\quad
\text{ as } N\ra +\infty,
\end{equation}
if $\sing{f}\not=0$, where\footnote{\ Here, except in the special case
  where all $f_j$ are linear, the singular series $\sing{f}$ is
  \emph{not} absolutely convergent (see below for more details on
  this; the problem is that $\nu_p(\uple{f})$ is only equal to $m$ on
  average over $p$, and not for all $p$ large enough, except if each
  $f_j$ is linear); the product is thus \emph{defined} as the limit of
partial products over primes $p\leq y$.}
\begin{equation}\label{eq-singf}
\sing{f}=\prod_p{
\Bigl(
1-\frac{\nu_p(\uple{f})}{p}
\Bigr)
\Bigl(1-\frac{1}{p}\Bigr)^{-m}},
\end{equation}
with $\nu_p(\uple{f})$ being now the number of $x\in \Zz/p\Zz$ such that
$f_j(x)=0$ for some $j$, $1\leq j\leq m$.
\par
The Hardy-Littlewood conjecture for a $k$-tuple $\uple{h}$ is 
equivalent with this conjecture for the primitive family
$$
\uple{f}=(X+h_1,\ldots,X+h_k)
$$
for which $\nu_p(\uple{h})$ as defined previously does coincide with
$\nu_p(\uple{f})$.
\par
\medskip
\par
Our goal is to study various averages of singular series, for which
there is undoubted arithmetic interest. A result of
Gallagher~\cite{gallagher} states that
\begin{equation}\label{eq-gallagher}
\lim_{h\ra +\infty}{
\frac{1}{h^k}\sums_{|\uple{h}|\leq h}{
\sing{h}
}
}=1,
\end{equation}
for any fixed $k$, as $h\ra +\infty$, where $|\uple{h}|=\max{h_i}$ and
$\sums$ restricts to $k$-tuples with distinct components.  This
property was used by Gallagher himself to understand the behavior of
primes in short intervals (see also the recent work by Montgomery and
Soundararajan~\cite{mont-sound}), and it is also important the
remarkable results of Goldston, Pintz and Y\i ld\i r\i m concerning
small gaps between primes (see~\cite{gpy} or the survey~\cite{b959}).
\par
Our first question is to ask about finer aspects of the distribution
of $\sing{h}$. To apply the method of moments, we first prove the
following:

\begin{theorem}\label{pr-gallagher-moments}
  Let $k\geq 1$ be fixed. For any complex number $m\in\Cc$ with
  $\Reel(m)\geq 0$, there exists a complex number $\mu_k(m)$ such that
$$
\lim_{h\ra +\infty}{
\frac{1}{h^k}\sums_{|\uple{h}|\leq h}{
\sing{h}^m
}
}=\mu_k(m).
$$
\par
Moreover, for $m$, $k\geq 1$ both integers, we have the symmetry
property
\begin{equation}\label{eq-symmetric}
\mu_k(m)=\mu_m(k)\ ;
\end{equation}
in addition, we have $\mu_1(m)=1$ for all integers $m\geq 1$, and
hence $\mu_k(1)=1$ for all $k\geq 1$.
\end{theorem}

The last statement ($\mu_k(1)=1$) is of course Gallagher's
theorem~(\ref{eq-gallagher}); our proof is not intrinsically
different, but maybe more enlightening.  These results are in fact
quite straightforward, and only the final symmetry in $k$ and $m$ is
maybe surprising. However, its origin is not particularly mysterious:
it is a ``local'' phenomenon, and it can be guessed
from~(\ref{eq-ktuple}) by a formal computation.
\par
We will also find estimates for the size of the moments which are good
enough to imply the existence of a limiting distribution of $\sing{h}$
for $k$-tuples ($k$ fixed):

\begin{theorem}\label{pr-limiting}
  Let $k\geq 1$ be fixed. There exists a probability law $\nu_k$ on
  $\Rr^+=[0,+\infty[$ such that $\sing{h}$, for $\uple{h}$ with
  $|\uple{h}|\leq h$ and $h\ra+\infty$, becomes equidistributed with
  respect to $\nu_k$, or equivalently
$$
\lim_{h\ra +\infty}{\frac{1}{h^k}
\sums_{|\uple{h}|\leq h}{f(\sing{h})}
}= \int_{\Rr^+}{f(t)d\nu_k(t)}
$$
for any bounded continuous function on $\Rr$.
\end{theorem}

The second question we explore is the generalization to other prime
patterns of the result of Gallagher (based on~(\ref{eq-gallagher}))
that shows that a uniform version of the prime $k$-tuple conjecture
implies that for a fixed $\lambda>0$, the distribution of
$\pi(x+\lambda \log x)-\pi(x)$ is close to a Poisson distribution of
parameter $\lambda$ as $x\ra +\infty$, i.e., it implies that
\begin{equation}\label{eq-gallagher-poisson}
\frac{1}{N}|\{n\leq N\,\mid\, \pi(n+h)-\pi(n)=m\}|
\ra e^{-\lambda}\frac{\lambda^m}{m!},\quad\quad\text{as } 
N\ra +\infty,
\end{equation}
for any integer $m\geq 0$. It turns out that, indeed, under a general
uniform version of the Bateman-Horn conjecture, for \emph{any} fixed
primitive family $\uple{f}$, the number of $\uple{f}$-prime seeds in
short intervals of ``fair'' length (i.e., intervals around $n$ in
which~(\ref{eq-bh}) predicts that, on average, there should be a fixed
number of $\uple{f}$-prime seeds) \emph{always} follows a Poisson
distribution. As for the symmetry property of the higher moments for
the singular series related to $k$-tuple conjecture, this turns out to
depend primarily on local identities, but we found this rigidity of
patterns to be quite surprising at first sight. Precisely:

\begin{theorem}\label{pr-poisson}
  Assume that the Bateman-Horn conjecture holds uniformly for all
  primitive families with non-zero singular series, in the sense that
\begin{equation}\label{eq-bh-unif}
  \pi(N;\uple{f})
  =\frac{1}{\peg(\uple{f})}\sing{f}\frac{N}{(\log N)^{m}}
  \Bigl(1+O\Bigl(\frac{c(\uple{f})^{\eps}}{\log N}\Bigr)\Bigr)
\end{equation}
holds for all primitive families $\uple{f}$, all $\eps>0$, and all
$N\geq 2$, where
$$
c(\uple{f})=\sum_{1\leq j\leq m}{H(f_j)},\quad\quad
H(a_0+a_1X+\cdots +a_dX^d) =\max_i |a_i|,
$$
and the implied constant depends at most on the degrees of the
elements of $\uple{f}$ and on $\eps$.
\par
Let $\uple{f}$ be a fixed primitive family with $\sing{f}\not=0$. For
$N\geq 1$, let
$$
\delta(N,\uple{f})=\frac{\peg(\uple{f})}{\sing{\uple{f}}}(\log N)^m.
$$
\par
Then for any $\lambda>0$ and any integer $r\geq 0$, we have
$$
\lim_{N\ra +\infty}{
\frac{1}{N}
|\{
n\leq N\,\mid\,
\pi(n+\lambda\delta(N,\uple{f});\uple{f})-\pi(n;\uple{f})=r
\}|
}=e^{-\lambda}\frac{\lambda^r}{r!}.
$$
\par
In other words, for $N$ large, the number of $\uple{f}$-prime
seeds in an interval around $N\geq 1$ of length $\lambda(\log
N)^m$ is asymptotically distributed like a Poisson random variable
with mean given by $\sing{f}\peg(\uple{f})^{-1}\lambda$.
\end{theorem}

The final purpose of this paper is to emphasize the fact that
Theorems~\ref{pr-gallagher-moments} and~\ref{pr-poisson} are special
cases of the problem of computing the average of some families of
values of Euler products, and (because here the Euler products are
absolutely convergent or almost so) the outcome is consistent with the
heuristic that the $p$-factors are \emph{independent random
  variables}, so the average of the Euler product is the product of
``local'' averages.  All this is a fairly common theme in analytic
number theory, but our presentation is maybe more systematic than
usual. The works of Granville-Soundararajan~\cite{granville-sound} and
Cogdell-Michel~\cite{cogdell-michel} also present this point of view
very successfully for values of certain families of $L$-functions at
the edge of the critical strip, and Y.  Lamzouri~\cite{lamzouri} has
developed this type of ideas in a quite general context.  Although
this is not really relevant from the point of view of singular series,
we just mention that Euler products built of local averages still make
sense inside the critical strip for many families of $L$-functions,
and are closely related to their distribution (as one can see, e.g.,
from the work of Bohr and Jessen~\cite{bj} for the Riemann zeta
function). On the critical line, ``renormalized'' Euler products still
occur in the moment conjectures for $L$-functions (see,
e.g.,~\cite{keating-snaith}), although other factors (conjecturally
linked to Random Matrices) also appear.

\par
\medskip 
\par
In the next section, we state in probabilistic terms a general result
on averages of random Euler products. Then we use it to prove
Theorem~\ref{pr-gallagher-moments} and
Theorem~\ref{pr-limiting} in Sections~\ref{sec-moments-singular}
and~\ref{sec-bound}. In Section~\ref{sec-poisson}, we prove
Theorem~\ref{pr-poisson}. 
\par
\medskip 
\par
\textbf{Notation.} As usual, $|X|$ denotes the cardinality of a set.
By $f\ll g$ for $x\in X$, or $f=O(g)$ for $x\in X$, where $X$ is an
arbitrary set on which $f$ is defined, we mean synonymously that there
exists a constant $C\geq 0$ such that $|f(x)|\leq Cg(x)$ for all $x\in
X$. The ``implied constant'' is any admissible value of $C$. It may
depend on the set $X$ which is always specified or clear in
context. On the other hand, $f\sim g$ as $x\ra x_0$ means $f/g\ra 1$
as $x\ra x_0$. 
\par
We use standard probabilistic terminology: a probability space
$(\Omega,\Sigma,\proba)$ is a triple made of a set $\Omega$ with a
$\sigma$-algebra and a measure $\proba$ on $\Sigma$ with
$\proba(\Omega)=1$. A random variable is a measurable function
$\Omega\ra \Rr$ (or $\Omega\ra \Cc$), and the expectation $\expect(X)$
on $\Omega$ is the integral of $X$ with respect to $\proba$ when
defined. The law of $X$ is the measure $\nu$ on $\Rr$ (or $\Cc$)
defined by $\nu(A)=\proba(X\in A)$. If $A\subset \Omega$, then
$\charfun_{A}$ is the characteristic function of $A$.
\par
For $k$-tuples $\uple{h}=(h_1,\ldots, h_k)$, we recall that
$|\uple{h}|=\max(|h_i|)$. When different values of $k$ can occur, we
sometimes write $|\uple{h}|_k$ to indicate the number of components of
$\uple{h}$, in particular a sum such as
$$
\sum_{|\uple{h}|_k\leq h}{a(\uple{h})}
$$
is a sum over $k$-tuples (of positive integers) with components $\leq
h$. 

\section{A probabilistic statement}
\label{sec-proba}

We assume given a probability space $(\Omega,\Sigma,\proba)$, and two
sequences of random variables 
$$
X_p,Y_p\,:\, \Omega\ra \Cc
$$
which are indexed by prime numbers.  
\par
We assume that $(Y_p)$ is an independent sequence; recall that this
means that
$$
\proba(Y_{p_1}\in A_1, \ldots, Y_{p_k}\in A_k)=
\prod_{1\leq i\leq k}{\proba(Y_{p_i}\in A_i)}
$$
for all choices of finitely many distinct primes $p_1$, \ldots, $p_k$,
and all measurable sets $A_i\subset \Cc$, and that a consequence is
that (when the expectation makes sense), we have
$$
\expect(Y_{p_1}\cdots Y_{p_k})=\expect(Y_{p_1})\cdots
\expect(Y_{p_k}).
$$
\par
We now extend the family to all integers by denoting
$$
X_q=\prod_{p\mid q}{X_p},\quad\quad
Y_q=\prod_{p\mid q}{Y_p},
$$
for any squarefree integer $q\geq 1$, and $X_q=Y_q=0$ if $q\geq 1$ is
not squarefree.
\par
We will consider the behavior of the random Euler products
$$
Z_X=\prod_{p}{(1+X_p)},\quad\quad Z_Y=\prod_{p}{(1+Y_p)}
$$
and in particular their expectations $\expect(Z_X)$ and
$\expect(Z_Y)$.
\par
For this purpose, we assume that the products converge
absolutely (almost surely). More precisely, expand formally
$$
\prod_{p}{(1+X_p)}=\sumb_{q\geq 1}{X_q},
$$
where $\sumb$ restricts the sum to squarefree numbers.  Then we assume
that
\begin{equation}\label{eq-remainder1}
\sumb_{q>x}{|X_q|}\leq R_X(x)
\end{equation}
where $R_X(x)$ is an integrable non-negative random variable such that
$R_X(x)\ra 0$ almost surely as $x\ra +\infty$. It then follows that
$Z_X$ is almost surely an absolutely convergent infinite product.
\par
We moreover assume that the product
\begin{equation}\label{eq-assump}
\prod_{p}{(1+|\expect(Y_p)|)}
\end{equation}
converges (absolutely). By independence of the $(Y_p)$, we know that
$$
|\expect(Y_q)|=\Bigl|\expect\Bigl(\prod_{p\mid q}{Y_p}\Bigr)\Bigr|
=\prod_{p\mid q}{|\expect(Y_p)|}
$$
and so expanding again in series, we obtain that
\begin{equation}\label{eq-yq}
\sumb_{q\geq 1}{|\expect(Y_q)|}=
\sumb_{q\geq 1}{\prod_{p\mid q}{|\expect(Y_p)|}}=
\prod_{p}{(1+|\expect(Y_p)|)}<+\infty.
\end{equation}
\par
Our goal is to show that if $(X_p)$ is distributed ``more or less''
like $(Y_p)$, but without being independent, the expectation of $Z_X$
is close to
$$
\prod_p{(1+\expect(Y_p))}.
$$
In particular, we will typically have $(X_p)$ depend on another
parameter (say $h$), in such a way that $X_{p,h}$ converges in law to
$Y_p$ (which will remain fixed) when $h\ra +\infty$, and this will
lead to the relation
$$
\lim_{h\ra +\infty} \expect\Bigl(\prod_p{(1+X_{p,h})}\Bigr)=
\prod_p{(1+\expect(Y_p))}
$$
in a number of situations. We interpret this as saying that (when
applicable) the average of the Euler product $Z_X$ is obtained ``as
if'' the factors were independent, and taking the product of the local
averages $1+\expect(Y_p)$ of the ``model'' random variables defining
$Z_Y$.
\par
Here is the precise (and almost tautological) ``finitary'' statement
from which applications will be derived.

\begin{proposition}\label{pr-euler}
Let $(X_p)$, $(Y_p)$ be as above.
Then for any choice of the auxiliary parameter $x>0$, we have
$$
\expect(Z_X)= \prod_{p}{(1+\expect(Y_p))}+O\Bigl(\expect(R_X(x))
+\sumb_{q\leq x}{|\expect(X_q-Y_q)|}
+\sumb_{q>x}{|\expect(Y_q)}|\Bigr),
$$
where the implied constant is absolute, and in fact has modulus at
most $1$.
\end{proposition}

\begin{proof}
  This more or less proves itself: for any $x\geq 1$, write first
$$
\prod_{p}{(1+X_p)}=\sumb_{q\geq 1}{X_q}=
\sumb_{q\leq x}{X_q}+\sumb_{q>x}{X_q},
$$
then use~(\ref{eq-remainder1}) to estimate the second term, and take
the expectation, which leads to
$$
\expect(Z_X)=\sum_{q\leq x}{\expect(X_q)}+O(\expect(R_X(x))).
$$
\par
Next, we insert $Y_q$ by writing $X_q=Y_q+(X_q-Y_q)$, getting
$$
\expect(Z_X)=\sumb_{q\leq x}{\expect(Y_q)}+
\sumb_{q\leq x}{\expect(X_q-Y_q)}+
O(\expect(R_X(x)))
$$
and then use
$$
\sumb_{q\leq x}{\expect(Y_q)}=
\sumb_{q\geq 1}{\expect(Y_q)}+O\Bigl(\sumb_{q>x}{|\expect(Y_q)|}\Bigr)=
\prod_p{(1+\expect(Y_p))}+O\Bigl(\sumb_{q>x}{|\expect(Y_q)|}\Bigr),
$$
to conclude the proof.
\end{proof}

\begin{remark}\label{rm-alpha}
  Observe that by~(\ref{eq-yq}), the last term in the remainder tends
  to zero as $x\ra +\infty$. Moreover, if $R_X(x)$ is dominated by an
  integrable function as $x\ra +\infty$, the assumption that
  $R_X(x)\ra 0$ almost surely implies that the first term also tends
  to zero. Thus to conclude in practical applications, one needs to
  control the middle term.
\par
In terms of the ``extra'' parameter $h$ mentioned before the statement
of the proposition, we may typically hope for uniform estimates for
$\expect(R_X(x))$, in terms of $h$, say
$$
\expect(R_X(x))\ll h^{\alpha}x^{-\beta},\quad\quad
\alpha,\ \beta>0;
$$
if we also have a bound of the type
\begin{equation}\label{eq-sieve-axiom}
\expect(X_q)=\expect(Y_q)+O(q^{\gamma}h^{-\delta}),\quad\quad
\gamma,\ \delta>0,
\end{equation}
(or if this holds on average over $q<x$, which may often be easier to
prove, as is the case for the error term in the prime number theorem,
as shows the Bombieri-Vinogradov theorem), this leads to a remainder
term which is
$$
\ll h^{\alpha}x^{-\beta}+x^{1+\gamma}h^{-\delta}+\eps(x)
$$
with $\eps(x)\ra 0$ as $x\ra +\infty$, uniformly in $h$. 
Then we can conclude that 
\begin{equation}\label{eq-converge}
\lim_{h\ra +\infty}{\expect(Z_X)}= \prod_{p}{(1+\expect(Y_p))}
\end{equation}
by choosing $x$ suitably as a function of $h$, \emph{provided} we have
$$
\frac{\alpha}{\beta}<\frac{\delta}{\gamma+1}.
$$
\par
We will see this in action concretely in the next sections. Notice
that if $\alpha$ can be chosen arbitrarily small (i.e., $R_X(x)$ is
bounded almost uniformly in terms of $h$), then this condition can be
met.
\end{remark}

\begin{remark}\label{rm-interp}
  If we assume, instead of~(\ref{eq-assump}), that the product of
  $1+\expect(|Y_p|)$ converges, which is stronger, it follows that
  $\sum{|Y_p|}<+\infty$ almost surely (its expectation being finite),
  and hence the infinite product defining $Z_Y$ converges absolutely
  almost surely. Also, since we have
$$
\expect\Bigl(\prod_{p\leq P}{(1+Y_p)}\Bigr)=
\prod_{p\leq P}{(1+\expect(Y_p))}
$$
for all $P$, we would obtain 
$$
\expect(Z_Y)=\prod_{p}{(1+\expect(Y_p))}.
$$
provided $Z_Y$ converges dominatedly, for instance.  This formula is
also valid if $Y_p\geq 0$, by the monotone convergence theorem.  It
provides an interpretation of the right-hand side
of~(\ref{eq-converge}).
\end{remark}

\section{Moments of singular series for the $k$-tuple conjecture}
\label{sec-moments-singular}

In this section, we prove Theorem~\ref{pr-gallagher-moments},
which includes in particular Gallagher's theorem, in a way which may
seem somewhat complicated but which clarifies the result.
\par
We first assume an integer $k\geq 1$ to be fixed. We
rewrite~(\ref{eq-singular}) as
$$
\sing{h}=\prod_{p}{\Bigl(
1+\frac{p^k-\nu_p(\uple{h})p^{k-1}-(p-1)^k}{(p-1)^k}
\Bigr)}.
$$
\par
It is therefore natural to define
$$
a(p,\nu)=\frac{p^k-\nu p^{k-1}-(p-1)^k}{(p-1)^k}
$$
for all primes $p$ and real numbers $\nu$, $0<\nu\leq p$ (omitting the
dependency on $k$). We then define $a_m(p,\nu)$, for $m\in\Cc$ with
$\Reel(m)\geq 0$, by requiring that
$$
1+a_m(p,\nu)=(1+a(p,\nu))^m,
$$
with the convention $0^{m}=0$ if $\Reel(m)=0$; the condition $\nu\leq
p$ implies that $1+a(p,\nu)\geq 0$, so this is well-defined indeed.
(If we assume $\nu<p$, we may extend this to all $m\in\Cc$).
\par
We first need a technical lemma.

\begin{lemma}\label{lm-apnu}
  For $m\in\Cc$ with $\Reel(m)\geq 0$, write $m^+=0$ if $\Reel(m)<1$,
  and $m^+=m-1$ otherwise.  For all $p$ prime and $\nu$ with $1\leq
  \nu\leq \min(p,k)$, we have
\begin{align}
a_m(p,k)&\ll
\frac{|m|}{p^2}\Bigl(1+O\Bigl(\frac{1}{p^2}\Bigr)\Bigr)^{m^+},\\
a_m(p,\nu)&\ll
\frac{|m|}{p}\Bigl(1+O\Bigl(\frac{1}{p}\Bigr)\Bigr)^{m^+},\quad\text{
  if } 1\leq \nu<k,
\end{align}
where the implied constants depend only on $k$.
\end{lemma}

\begin{proof}
  Notice first that, in the stated range, we have
\begin{align*}
a(p,k)&\ll p^{-2},\\
a(p,\nu)&\ll p^{-1},\quad\text{ if } 1\leq \nu<k,
\end{align*}
where the implied constants depend only on $k$,
and then write
$$
a_m(p,\nu)=(1+a(p,\nu))^m-1=
ma(p,\nu)\int_0^{1}{
(1+ta(p,\nu))^{m-1}dt
}
$$
and estimate directly.
\end{proof}

We are now going to prove Theorem~\ref{pr-gallagher-moments}. Fix
$h\geq 1$ (though $h$ will tend to infinity at the end). We first
interpret the $m$-th moment of the singular series in probabilistic
terms, then introduce the source of its limiting value in the
framework of the previous section.
\par
Consider the finite set (again, depending on $k$)
$$
\Omega_1=\{\uple{h}=(h_i)\,\mid\, 1\leq h_i\leq h,\ h_i\text{ distinct}\},
$$
with the normalized counting measure. Denoting $h_k^*=|\Omega_1|$,
notice that
\begin{equation}\label{eq-compar}
h_k^*=h^k(1+O(h^{-1}))
\end{equation}
for $h\geq 1$, the implied constant depending only on $k$. We will
denote by $\expect_1$ and $\proba_1$ the expectation and probability
for this discrete space. So we have, for instance, that
$$
\proba_1(\nu_p=\nu)=\frac{1}{h_k^*}|
\{
\uple{h}\in\Omega_1\,\mid\, \nu_p(\uple{h})=\nu
\}
|.
$$
\par
Our goal is to find the limit as $h\ra +\infty$ of the average
$$
\frac{1}{h^*_k}
\sum_{\stacksum{|\uple{h}|\leq h}{h_i\text{ distinct}}}{
\sing{\uple{h}}^m}=\expect_1(\sing{\uple{h}}^m)
$$
(notice that, by~(\ref{eq-compar}), if the limit exists, it is also
the limit of 
$$
\frac{1}{h^k}
\sum_{\stacksum{|\uple{h}|\leq h}{h_i\text{ distinct}}}{
\sing{\uple{h}}^m},
$$
as $h\ra +\infty$).
\par
We write $X_p(\uple{h})=a(p,\nu_p(\uple{h}))$ and
$X_p(m,\uple{h})=a_m(p,\nu_p(\uple{h}))$, so that
$$
\prod_{p}{(1+X_p(m,\uple{h}))}=\sing{\uple{h}}^m
$$
by construction.
\par
Now consider a second space
$$
\Omega_2=\prod_p{(\Zz/p\Zz)^k}
$$
with the product measure of the probability counting measures on each
factor. We denote by $\omega=(\uple{h}_p)_p$ the elements of
$\Omega_2$. To avoid confusion with $\nu_p$ defined for
$\uple{h}\in\Omega_1$, we introduce the random variables
$$
\rho_p\ :\ \begin{cases}
\Omega_2\ra \{1,\ldots,k\}\\
\omega=(\uple{h}_p)_p\mapsto \text{number of distinct $h_i$ in $\Zz/p\Zz$},
\end{cases}
$$
which satisfy $1\leq \rho_p\leq \min(k,p)$.
\par
We can now define ``random'' singular series using $\Omega_2$, writing
$Y_p=a(p,\rho_p)$ and considering the Euler product
$$
\prod_p{(1+Y_p)},
$$
and similarly with $Y_p(m)=a_m(p,\rho_p)$ and
$$
\prod_p{(1+Y_p(m))}=\Bigl(\prod_p{(1+Y_p)}\Bigr)^m.
$$
\par
We denote by $\proba_2$ and $\expect_2$ the probability and
expectation for this space.
By construction of $\Omega_2$, the random variables $(\rho_p)$ are
independent, and so are the $(Y_p)$, and the $(Y_p(m))$ for a given
$m$. Note also that the components $\uple{h}_p$ are equidistributed:
for any prime $p$ and any $a\in (\Zz/p\Zz)^k$, we have
\begin{equation}\label{eq-law}
\proba_2(\uple{h}_p=a)=\frac{1}{p^k}.
\end{equation}
\par
We now use Proposition~\ref{pr-euler} to compare the average
$\expect_1(\sing{h}^m)$ with 
$$
\prod{\expect_2((1+Y_p)^m)}.
$$
\par
Although this proposition is phrased with a single probability space
$\Omega$ on which both Euler vectors are defined, this is not a
serious issue and the statement remains valid, provided the
expectations are suitably subscripted and one writes
$$
\Bigl|\expect_1(X_q(m))-\expect_2(Y_q(m))\Bigr|
$$
on the right-hand side instead of
$|\expect(X_q(m)-Y_q(m))|$.\footnote{\ We could also simply consider
  $\Omega=\Omega_1\times \Omega_2$ with the product measure, or
  equivalently (and maybe more elegantly) assume that we start with
  some space $\Omega$ and two vectors $(X_p)$, $(Y_p)$, distributed
  according to the prescription of $\Omega_1$ and $\Omega_2$
  respectively, i.e., with
\begin{align*}
  \proba(X_p=a)&=\frac{1}{h_k^*}|\{\uple{h}\in\Omega_1\,\mid\,
  a(p,\nu_p(\uple{h}))=a \}|,\\
  \proba(Y_p=a)&=\frac{1}{p^k}|\{\uple{h}\in (\Zz/p\Zz)^k\,\mid\,
  a(p,\rho_p(\uple{h}))=a \}|.
\end{align*}
}
\par
We start by estimating the tail $R(x)=R_{X(m)}(x)$ of the Euler
product defining $\sing{h}^m$. In keeping with probabilistic
conventions, we omit the argument $\uple{h}\in\Omega_1$ in many
places. Denoting
$$
\Delta(\uple{h})=\Bigl|\prod_{i<j}{(h_i-h_j)}\Bigr|\geq 1,
$$
and noting that $\nu_p=k$ unless $p\mid \Delta$, we have from
Lemma~\ref{lm-apnu} the bound
$$
|X_p(m)|\ll |m|\Bigl(1+O\Bigl(\frac{(p,\Delta)}{p^2}\Bigr)\Bigr)^{m^+}
(p,\Delta)p^{-2}
$$
for some $C>0$ (depending only on $k$) and all $\uple{h}$, $m$ (with
$\Reel(m)\geq 0$) and $p$, the implied constant depending only on $k$
(this justifies, in particular, the convergence of the Euler product
$Z_X$ for every $\uple{h}$). Hence, taking the product over $p\mid q$
for a squarefree integer $q$, we get
$$
|X_q(m)|\leq
(|m|B)^{\omega(q)}(q,\Delta)q^{-2}
\prod_{p\mid q}{\Bigl(1+C\frac{(p,\Delta)}{p^2}\Bigr)^{m^+}}
$$
for some constants $B>0$ and $C\geq 0$ depending only on $k$.  
Since $\Delta$ is bounded by
\begin{equation}\label{eq-discrim}
|\Delta|\leq (2h)^{k^2},
\end{equation}
a standard computation with sums of multiplicative functions leads to
$$
\sumb_{q>x}{|X_q(m)|}\ll x^{-1}(\log 2hx)^{D}
$$
for $x\geq 2$ and some constant $D\geq 0$, depending on $k$ and $m$.
\par
The next step is to justify the analogue of the convergence
of~(\ref{eq-assump}); more precisely, we have
\begin{equation}\label{eq-integrable}
\prod_{p}{(1+\expect_2(|Y_p(m)|))}<+\infty.
\end{equation}
\par
Indeed, Lemma~\ref{lm-apnu} leads to
$$
\expect_2(|Y_p(m)|)\ll p^{-2}+p^{-1}\proba_2(\rho_p<k)\ll p^{-2}
$$
for $p\geq 2$, where the implied constant depends on $k$ and $m$,
since it is clear that we have
\begin{equation}\label{eq-proba-rhop}
\proba_2(\rho_p<k)\leq \frac{k(k-1)}{2p}
\end{equation}
for all primes $p$ and $k\geq 1$ (write that the event $\{\rho_p<k\}$
is the union -- not necessarily disjoint -- of the $k(k-1)/2$ events
$h_i=h_j$ with $i\not=j$, each of which has probability $1/p$ by
uniform distribution~(\ref{eq-law})). By independence, we then also
get
\begin{equation}\label{eq-expyq}
\expect_2(|Y_q(m)|)\leq A^{\omega(q)}q^{-2}.
\end{equation}
for all squarefree integers $q$ and some constant $A\geq 1$, which
depends only on $k$ and $m$.
\par
Finally, it remains to estimate $\expect_1(X_q(m))-\expect_2(Y_q(m))$.
We claim that, for any $a\in\Cc$, we have
\begin{equation}\label{eq-equidist}
\proba_1(X_q(m)=a)=\Bigl(1+O\Bigl(\frac{q}{h}\Bigr)\Bigr)
\proba_2(Y_q(m)=a)+O\Bigl(\frac{k^{\omega(q)}}{h}\Bigr)
\end{equation}
where the implied constants depend only on $k$. Assuming this, and
noting that $X_q(m)$ and $Y_q(m)$ take the same finitely many values
(at most $k^{\omega(q)}$ distinct values, which are
$$
\ll \frac{F^{\omega(g)}}{q}
$$
where the implied constant and $F$ depend on $m$ and $k$),
it follows that
$$
\expect_1(X_q(m))=\Bigl(1+O\Bigl(\frac{q}{h}\Bigr)\Bigr)
\expect_2(Y_q(m))+O\Bigl(\frac{G^{\omega(q)}}{h}\Bigr),
$$
where $G$ depends on $m$ and $k$, leading in turn to
$$
\Bigl|\expect_1(X_q(m))-\expect_2(Y_q(m))\Bigr|\ll
\frac{q}{h}\expect_2(|Y_q(m)|)+\frac{G^{\omega(q)}}{h}\ll
\frac{E^{\omega(q)}}{h}
$$
(see~(\ref{eq-expyq})), where the implied constant depends only on $k$
and $m$, as does $E$.
\par
Summing over $q<x$, it then follows from Proposition~\ref{pr-euler}
that
\begin{multline*}
\frac{1}{h_k^*}\sums_{\uple{h}}{\sing{h}^m}=
\expect_1\Bigl(\prod_{p}{(1+X_p(m))}\Bigr)=
\prod_p{(1+\expect_2(Y_p(m)))}+\\
O\Bigl(
xh^{-1}(\log 2hx)^B+x^{-1}(\log 2hx)^D
\Bigr)
\end{multline*}
for some $B$ depending on $k$ and $m$. Choosing for instance
$x=h^{1/2}$ leads to the existence of the $m$-th moment of singular
series, with limiting value given by
\begin{equation}\label{eq-moment}
\mu_k(m)=
\prod_{p}{(1+\expect_2(Y_p(m)))}=
\prod_{p}{\Bigl(1-\frac{1}{p}\Bigr)^{-km}
\Bigl\{\frac{1}{p^k}
\sum_{\uple{h}\in(\Zz/p\Zz)^k}{
\Bigl(1-\frac{\rho_p(\uple{h})}{p}
\Bigr)^m}\Bigr\}
}.
\end{equation}
\par
It only remains to prove~(\ref{eq-equidist}). Note that this is
clearly an expression of quantitative equidistribution (or convergence
in law) of $X_q$ to $Y_q$ as $h\ra +\infty$.\footnote{\ It can also be
interpreted as a form of ``sieve axiom''.}
\par
The proof is quite simple.  First of all, given arbitrary integers
$s_p$ with $p\mid q$, we have
\begin{align*}
  \proba_1(\nu_p(\uple{h})=s_p\text{ for } p\mid q)
  &=\frac{1}{h_k^*}\sums_{\stacksum{\nu_p(\uple{h})=s_p\text{ for }
      p\mid q}{|\uple{h}|
      \leq h}}{1}\\
  &=\frac{1}{h_k^*} \multsum_{\stacksum{\rho_p(\uple{h}_p)=s_p}
    {\uple{h}_p\in(\Zz/p\Zz)^k}} \sums_{\stacksum{|\uple{h}|\leq
      h}{\uple{h}\equiv \uple{h}_p\mods{p\mid q}}}{1}
\end{align*}
(where there are as many outer sums in the last line as there are
primes dividing $q$, and the last sum involves summation conditions
for all $p\mid q$). This inner sum is 
\begin{equation}\label{eq-mid-equidist}
\sums_{\stacksum{|\uple{h}|\leq h}{\uple{h}\equiv
    \uple{h}_p\mods{p\mid q}}}{1}
=\sum_{\stacksum{|\uple{h}|\leq h}{\uple{h}\equiv
    \uple{h}_p\mods{p\mid q}}}{1}
+O(h^{k-1})
\end{equation}
where the implied constant depends on $k$ (i.e., we now forget the
condition on $\uple{h}$ to have distinct components). Lattice-point
counting leads to
$$
\sum_{\stacksum{|\uple{h}|\leq h}{\uple{h}\equiv
    \uple{h}_p\mods{p\mid q}}}{1}=
\frac{h^k}{q^k}\Bigl(1+O\Bigl(\frac{q}{h}\Bigr)\Bigr)
$$
where the implied constant depends again only on $k$. In view of the
equidistribution of $\uple{h}_p$ for $(\uple{h}_p)_p\in \Omega_2$, 
we therefore derive from the above the following \emph{quantitative
  equidistribution} result:
\begin{equation}\label{eq-proba-nup}
\proba_1\bigl(\nu_p(\uple{h})=s_p\text{ for } p\mid q\bigr)
=\proba_2\bigl(\rho_p(\uple{h}_p)=s_p\text{ for } p\mid q\bigr)
\Bigl(1+O\Bigl(\frac{q}{h}\Bigr)
\Bigr)+O\Bigl(\frac{1}{h}\Bigr).
\end{equation}
\par
Now to derive~(\ref{eq-equidist}), we need only observe that $Y_q(m)$
and $X_q(m)$ are ``identical'' functions of $\rho_p$ and $\nu_p$
respectively (for $p\mid q$).  Hence~(\ref{eq-proba-nup})
implies~(\ref{eq-equidist}) by summing over all possible values of
$(s_p)_{p\mid q}$ leading to a given $a$, using the fact that there
are at most $k^{\omega(q)}$ such values (the latter being a very rough
estimate!).
\par
It remains to prove the symmetry property~(\ref{eq-symmetric}) to
finish the proof of Theorem~\ref{pr-gallagher-moments}.  We note
in advance that since $\sing{h}=1$ for all $1$-tuple $\uple{h}$, we
have $\mu_1(m)=1$ for all $m\geq 1$, and hence $\mu_k(1)=1$ for all
$k\geq 1$, which is Gallagher's result~(\ref{eq-gallagher}).
\par
The symmetry turns out to be true ``locally'', i.e., the $p$-factor of
the Euler products~(\ref{eq-moment}) defining $\mu_k(m)$ and
$\mu_m(k)$ coincide for all $p$ and integers $k$, $m\geq 1$.
\par
There are different ways to see this, and the following seems to
encapsulate the origin of the phenomenon. Given a finite set $F$
(which will be $\Zz/p\Zz$), consider the following obviously
symmetric expression of $m$ and $k$:
$$
\frac{1}{|F|^{m+k}} \dblsum_{\stacksum{\uple{x}\in F^m, \ \
    \uple{h}\in F^k}{\{x_i\}\cap \{h_j\}=\emptyset}}{1}
$$
(which is the probability, for the normalized counting measure on
$F^{k+m}$, that a pair of a $k$-tuple and an $m$-tuple, both of
elements of $F$, do not contain a common element). Then it can be
interpreted either as
\begin{align*}
\frac{1}{|F|^m}\sum_{\tau=1}^m{
  \sum_{\stacksum{\uple{x}\in F^m}{\rho(\uple{x})=\tau}}
    \frac{1}{|F|^k}
    \sum_{\stacksum{\uple{h}\in F^k}{\{h_j\}\cap
        \{x_i\}=\emptyset}}{1}}
&=\frac{1}{|F|^m}\sum_{\tau=1}^m{
  \sum_{\stacksum{\uple{x}\in F^m}{\rho(\uple{x})=\tau}}
    \Bigl(1-\frac{\tau}{|F|}\Bigr)^k}\\
&=\frac{1}{|F|^m}
\sum_{\uple{x}\in F^m}{
    \Bigl(1-\frac{\rho(\uple{x})}{|F|}\Bigr)^k}
\end{align*}
or (by the same computation with $m$ and $k$ reversed) as
$$
\frac{1}{|F|^k}
\sum_{\uple{h}\in F^k}{
    \Bigl(1-\frac{\rho(\uple{h})}{|F|}\Bigr)^m},
$$
(using $\rho(\cdot)$ to denote the number of distinct elements in $F$
of an $m$-tuple, then of a $k$-tuple). 
\par
Applied with $F=\Zz/p\Zz$, up to the symmetric factor $(1-1/p)^{-km}$
in~(\ref{eq-moment}), the first is the $p$-factor for $\mu_m(k)$, and
the second is the $p$-factor for $\mu_k(m)$, showing that they are
indeed equal.

\begin{remark}
Quantitatively, we have proved that 
$$
\sums_{|\uple{h}|\leq h}{\sing{h}^m}=
\mu_k(m)h^*_k+O(h^{k-1/2+\eps}),
$$
for any $\eps>0$, where the implied constant depends on $k$ and $m$.
For $m=1$, Montgomery and Soundararajan~\cite[(17), p.
593]{mont-sound} have obtained a more refined expansion with
contributions of size $h^{k-1}\log h$ and $h^{k-1}$, and error term of
size $h^{k-3/2+\eps}$.
\end{remark}

\begin{remark} The fact that $\mu_k(1)=1$ can be used to recover the
  combinatorial identities used by Gallagher~\cite[p. 7--8]{gallagher}
  instead of the probabilistic phrasing above. We review this for
  completeness: in order to prove $\mu_k(1)=1$, it suffices to show
  that the average of $a(p,\rho_p)$ is zero. We have
$$
\sum_{\uple{h}\in(\Zz/p\Zz)^k}{
a(p,\rho_p(\uple{h}))}
=\sum_{\nu=1}^p{a(p,\nu)|\{\uple{h}\in (\Zz/p\Zz)^k\,\mid\,
  \rho_p(\uple{h})=\nu \}|}
$$
and on the other hand, we have
$$
|\{\uple{h}\in (\Zz/p\Zz)^k\,\mid\, \rho_p(\uple{h})=\nu \}|
=\binom{p}{\nu} \stirling{k}{\nu},
$$
where $\stirling{k}{\nu}$ is the number of surjective maps from a set
with $k$ elements to one with $\nu$ elements\footnote{\ This is
  denoted $\sigma(k,\nu)$ in~\cite{gallagher}, and it is \emph{not}
  the standard notation, which would write $r!\stirling{k}{r}$
  instead.}; indeed, a $k$-tuple $\uple{h}$ with $\nu$ distinct values
is the same as a map $\{1,\ldots, k\}\ra \Zz/p\Zz$ with image of
cardinality $\nu$, i.e., the set of such tuples is the disjoint union
of those sets of surjective maps
$$
\{1,\ldots, k\}\ra I
$$
over $I\subset \Zz/p\Zz$ with order $\nu$.
\par
Therefore, Gallagher's result follows from the identity
$$
\sum_{\nu=1}^p{a(p,\nu)
\binom{p}{\nu} \stirling{k}{\nu}}=0
$$
which is proved in~\cite[p. 7]{gallagher}, and which we have therefore
reproved. Similarly, the identities
$$
\sum_{\nu=1}^p{\binom{p}{\nu} \stirling{k}{\nu}}=p^k,\quad
\sum_{\nu=1}^p{\nu\binom{p}{\nu} \stirling{k}{\nu}}=p^{k+1}-p(p-1)^k,
$$
of~\cite[p. 8]{gallagher} can be derived from the proof that the
$p$-factor for $\mu_k(1)$ is $1$.
\end{remark}

\begin{remark}
From~(\ref{eq-ktuple}), one can guess that $\mu_k(m)=\mu_m(k)$ for
$m\geq 1$ integer, by computing
$$
\sum_{|\uple{h}|\leq h}{\Bigl(\sum_{n\leq N}{
\prod_{1\leq i\leq k}\Lambda(n+h_i)}\Bigr)^m}
=\sum_{|\uple{h}|_k\leq h}\sum_{|\uple{n}|_m\leq N}{
\prod_{\stacksum{1\leq i\leq k}{1\leq j\leq m}}
{\Lambda(n_j+h_i)}}
$$
(where $\uple{n}$ is an $m$-tuple), which is a symmetric expression in
$\uple{n}$ and $\uple{h}$, except for the ranges of summation, and
which should be asymptotic to either $\mu_k(m)h^kN^m$ or $\mu_m(k)
h^kN^m$ by a uniform $k$-tuple conjecture.  In fact, the computation
we did amounts to doing the same argument locally (i.e., looking on
average over $\uple{h}$ at the distribution of integers such that, for
a fixed prime $p$, $n+h_1$,\ldots, $n+h_k$ are not divisible by $p$).
\par
This symmetry $\mu_k(m)=\mu_m(k)$, despite the simplicity of its
proof, is a very strong property, as pointed out to us by A.
Nikeghbali. Indeed, write $X_k=Z_{Y,k}$, the random variable given by
the random singular series. Since we have
$$
\mu_k(m)=\int_{\Rr^+}{t^md\nu_k(t)}=E(X_k^m),
$$
the symmetry implies that the sequence $(E(X_k^m))_k$, for a
\emph{fixed} value of $m$, is the sequence of moments of a probability
distribution of $[0,+\infty[$, which is a highly non-trivial property.
We refer to the survey~\cite{simon} of the classical theory
surrounding the ``moment problems'', noting that from Theorem~1 of
loc. cit. it follows that, for any fixed $m\geq 1$, we have
$$
\dblsum_{\stacksum{0\leq i\leq N}{0\leq j\leq N}}{
\alpha_i\bar{\alpha_j} \mu_{i+j}(m)
}> 0,\quad\quad
\dblsum_{\stacksum{0\leq i\leq N}{0\leq j\leq N}}{
\alpha_i\bar{\alpha_j} \mu_{i+j+1}(m)
}> 0,
$$
for any $N\geq 1$ and any complex numbers $(\alpha_i)\in \Cc^N-\{0\}$.
\par
It would be quite interesting to know what other types of natural
sequences of random variables (or probability distributions) satisfy
the relation $E(X_k^m)=E(X_m^k)$. One fairly general construction is
as follows (this was pointed out by A. Nikeghbali and P. Bourgade):
just take $X_n=Z^n$ for $Z$ a random variable such that all moments of
$Z$ exist, or a bit more generally, take a sequence $(X_n)$ of
positive random variables such that the $X_n^{1/n}$ are identically
distributed. But note that the variables we encountered are not of
this type.
\end{remark}

\begin{example}\label{ex-m2}
  Let $m=2$. We find (using the symmetry property) that the
  mean-square of $\sing{h}$ is given by
$$
\lim_{h\ra +\infty}{
\frac{1}{h^k}\sums_{|\uple{h}|\leq h}{
\sing{h}^2
}
}=\mu_k(2),
$$
where
$$
\mu_k(2)=\prod_{p}{\Bigl(
\Bigl(1-\frac{1}{p}\Bigr)
\Bigl(1-\frac{2}{p}\Bigr)^k+
\frac{1}{p}\Bigl(1-\frac{1}{p}\Bigr)^k
\Bigr)\Bigl(1-\frac{1}{p}\Bigr)^{-2k}}.
$$
\par
In particular, we find (using \url{Pari/GP} for instance):
\begin{gather*}
\mu_2(2)=2.300\ldots\quad\quad
\mu_3(2)=6.03294\ldots\\
\mu_4(2)=17.562\ldots\quad\quad
\mu_5(2)=55.255\ldots\\
\mu_6(2)=184.18\ldots
\end{gather*}
\par
Note that the second (and higher) moments increase quickly with $k$
(as proved in Proposition~\ref{pr-growth} in the next section).  This
is explained intuitively by the fact that $\sing{h}$ is often zero: for
instance, the $2$-factor of $\sing{h}$ is zero unless all $h_i$ are of
the same parity, which happens with probability $2^{1-k}$ only (see
Example~\ref{ex-zero} for a more precise estimate). For those, of
course, the $2$-factor is very large (equal to $2^{k-1}$).
\end{example}

\section{Growth and distribution of moments of singular series}
\label{sec-bound}

In this section, we will prove Theorem~\ref{pr-limiting}, using
the methods of moments. For this, we consider the problem (which has
independent interest) of determining the growth of $\mu_k(m)$.  We
look at the dependency on $m$ for fixed $k$, or equivalently the
dependency on $k$ for fixed $m$, by symmetry (as in
Example~\ref{ex-m2}). The result is that the moments grow just a bit
faster than exponentially.

\begin{proposition}\label{pr-growth}
For any fixed $k\geq 1$, we have
$$
\log \mu_k(m)=km\log\log 3m+O(m),\quad\quad\text{ for } m\geq 1,
$$
where the implied constant depends on $k$.
\end{proposition}

\begin{proof}
  We use the formula~(\ref{eq-moment}), written in the form
$$
\mu_k(m)=\prod_p{
\Bigl(1-\frac{1}{p}\Bigr)^{-km}
\expect_2\Bigl(\Bigl(1-\frac{\rho_p}{p}\Bigr)^m\Bigr).
}
$$
\par
We will prove first that
$$
\log \mu_k(m)\geq km\log\log 3m+O(m),
$$
for $m\geq 1$, with an implied constant depending on $k$, before
proving the corresponding upper bound.
\par
We start by checking that all terms in the Euler product are $\geq 1$,
i.e., for all primes $p$, all integers $k$ and all real numbers $m\geq
1$, we have
\begin{equation}\label{eq-geq1}
\expect_2\Bigl(\Bigl(1-\frac{\rho_p}{p}\Bigr)^m\Bigr)
\geq \Bigl(1-\frac{1}{p}\Bigr)^{mk}.
\end{equation}
\par
Indeed, by the symmetry between the $p$-factor for $\mu_k(1)$ and for
$\mu_1(k)$, we have
$$
\Bigl(1-\frac{1}{p}\Bigr)^k=
\expect_2\Bigl(1-\frac{\rho_p}{p}\Bigr),
$$
while raising to the $m$-th power and applying H\"older's inequality
gives 
$$
\Bigl(\expect_2\Bigl(1-\frac{\rho_p}{p}\Bigr)\Bigr)^m
\leq \expect_2\Bigl(\Bigl(1-\frac{\rho_p}{p}\Bigr)^m\Bigr).
$$
\par
From this we can bound $\mu_k(m)$ from below by any subproduct, and
we look at
$$
\mu_k^*(m)=
\prod_{p\leq m}{\Bigl(1-\frac{1}{p}\Bigr)^{-km}
\expect_2\Bigl(\Bigl(1-\frac{\rho_p}{p}\Bigr)^m\Bigr)}.
$$
\par
The probability that $\rho_p$ is $1$ is clearly equal to $p^{-(k-1)}$
(there are only $p$ $k$-tuples with this property). Hence we have 
crude lower bounds
$$
\expect_2\Bigl(\Bigl(1-\frac{\rho_p}{p}\Bigr)^m\Bigr)
\geq \frac{1}{p^{k-1}}\Bigl(1-\frac{1}{p}\Bigr)^k
$$
and
$$
\mu_k(m)\geq \mu_k^*(m)\geq 
\prod_{p\leq
  m}{\Bigl(1+\frac{1}{p-1}\Bigr)^{k(m-1)}\frac{1}{p^{k-1}}}.
$$
\par
The logarithm of this expression is easily bounded from below as
follows:
\begin{align*}
\log \mu_k(m)&\geq k(m-1)\sum_{p\leq m}{\log\Bigl(1+\frac{1}{p-1}\Bigr)}
-(k-1)\sum_{p\leq m}{\log p}\\
&=km\log\log 3m+O(m),
\end{align*}
for $m\geq 2$, the implied constant depending only on $k$, by standard
estimates, and we can incorporate trivially $m=1$ also.
\par
To prove the corresponding upper bound, we split the Euler
product~(\ref{eq-moment}) into two ranges: we write
$$
\mu_k(m)=\mu^{(1)}_k(m)\mu_k^{(2)}(m),
$$
where $\mu^{(1)}_k(m)$ is the product over primes $p<km$ (which
includes the range used for the lower bound), while $\mu^{(2)}_k(m)$
is the product over the other primes $p\geq km$. We will show that
$$
\log \mu_k^{(1)}(m)\leq km\log\log 3m+O(m),\quad \log
\mu_k^{(2)}(m)\ll \frac{m}{\log 2m},
$$
with implied constants depending on $k$, and this will conclude the
proof.
\par
We start with small primes, and simply bound the expectation of
$(1-\rho/p)^m$ by the trivial bound $1$; this leads to
$$
\log \mu_k^{(1)}(m)\leq -km \sum_{p<km}{\log\Bigl(1-\frac{1}{p}\Bigr)}
=km\log\log 3m+O(m),
$$
where the implied constant depends on $k$, again by standard
estimates.
\par
Next, we estimate $\mu_k^{(2)}(m)$ more carefully. The logarithm (say
$\mathcal{L}(x)$) of the product restricted to $km\leq p\leq x$ is
given by
$$
\mathcal{L}(x)=-km\sum_{km\leq p\leq x}{\log (1-p^{-1})} +\sum_{km\leq
  p\leq x}{\log \expect_2\Bigl(\Bigl(1-\frac{\rho_p}{p}\Bigr)^m\Bigr)}.
$$
\par
Using~(\ref{eq-proba-rhop}), we write first, for $p\geq km$, the upper
bound 
\begin{align*}
\expect_2\Bigl(\Bigl(1-\frac{\rho_p}{p}\Bigr)^m\Bigr)&\leq
\Bigl(1-\frac{k}{p}\Bigr)^m(1-\proba_2(\rho_p<k))+
\proba_2(\rho_p<k)\\
&=\Bigl(1-\frac{k}{p}\Bigr)^m+
\proba_2(\rho_p<k)\Bigl(1-\Bigl(1-\frac{k}{p}\Bigr)^m\Bigr)\\
&\leq \Bigl(1-\frac{k}{p}\Bigr)^m+
\frac{mk^2(k-1)}{2p^2}\\
&\leq
1-\frac{mk}{p}+\frac{m(m-1)}{2}\frac{k^2}{p^2}
+\frac{mk^2(k-1)}{2p^2},\\
&=1-\frac{mk}{p}+\frac{m^2k^2}{2p^2}+\frac{mA_k}{2p^2}
\end{align*}
(with $A_k=k^3-2k^2$) since
$$
1-mx\leq (1-x)^m\leq 1-mx+\frac{m(m-1)}{2}x^2\quad\quad
\text{ for } 0\leq x\leq 1,\ m\geq 1.
$$
\par
Moreover, we have $\log(1-x)\leq -x-x^2/2$ for $0\leq x<1$, and hence
after some rearranging, we obtain
\begin{multline*}
\log \expect_2\Bigl(\Bigl(1-\frac{\rho_p}{p}\Bigr)^m\Bigr)
\leq 
-\frac{mk}{p}+\frac{m^2k^2}{2p^2}+\frac{mA_k}{2p^2}
-\frac{1}{2}
\Bigl(
\frac{mk}{p}-\frac{m^2k^2}{2p^2}-\frac{mA_k}{2p^2}
\Bigr)^2\\
=-\frac{mk}{p}+
\frac{m^3k^2}{p^3}
-\frac{m^4k^4}{8p^4}
+\frac{mA_k}{2p^2}
-\frac{m^2kA_k}{2p^3}
-\frac{m^2A_k^2-2m^3k^2A_k}{8p^4},
\end{multline*}
the terms involving $(m^2k^2)/(2p^2)$ having cancelled out. 
\par
Summing over $km\leq p\leq x$, we can let $x$ go to infinity in all
but the first resulting term since they define convergent series;
bounding the tail by
$$
\sum_{p>km}{\frac{1}{p^{\sigma}}}\ll (km)^{1-\sigma}(\log 2km)^{-1},
$$
leads to
$$
\sum_{km\leq p\leq x}{
\log \expect_2\Bigl(\Bigl(1-\frac{\rho_p}{p}\Bigr)^m\Bigr)
}
\leq -km\sum_{km\leq p\leq x}{\frac{1}{p}}
+O\Bigl(\frac{m}{\log 2m}\Bigr)
$$
for all $m$ and $x\geq km$, where the implied constant depends on $k$.
Finally,
$$
\log\mathcal{L}(x)\leq -km\sum_{km<p\leq x}{
\Bigl(\frac{1}{p}+\log\Bigl(1-\frac{1}{p}\Bigr)\Bigr)
}
+O\Bigl(\frac{m}{\log 2m}\Bigr),
$$
and since $p^{-1}+\log(1-p^{-1})$ defines an absolutely convergent
series with tail (for $p>y$) decreasing like $y^{-1}(\log y)^{-1}$, we
obtain the desired bound for
$$
\log \mu_k^{(2)}(m)=\lim_{x\ra +\infty}{\mathcal{L}(x)}.
$$
\end{proof}

The existence of a limiting distribution
(Theorem~\ref{pr-limiting}) is an easy consequence of this.

\begin{corollary}\label{cor-distrib}
  Let $k\geq 1$ be a fixed integer. As $h$ goes to infinity, the
  singular series $\sing{h}$ for $\uple{h}\in \Omega_1$, i.e., such
  that $|\uple{h}|\leq h$, converges in law to the random singular
  series
$$
Z_Y=Z_{Y,k}=\prod_{p}{\Bigl(1-\frac{1}{p}\Bigr)^{-k}
\Bigl(1-\frac{\rho_p}{p}\Bigr)}
$$
on $\Omega_2$. In other words, there exists a probability law $\nu_k$
on $[0,+\infty[$, which is the law of $Z_Y$, such that $\sing{h}$, for
$|\uple{h}|\leq h$, becomes equidistributed with respect to $\nu_k$, or
equivalently
$$
\lim_{h\ra +\infty}{\frac{1}{h^k}
\sums_{|\uple{h}|\leq h}{f(\sing{h})}
}= \int_{\Rr^+}{f(t)d\nu_k(t)}
$$
for any bounded continuous function on $\Rr$. Moreover we have
\begin{equation}\label{eq-true-moment}
\mu_k(m)=\expect_2(Z_Y^m)=\int_{\Rr^+}{t^md\nu_k(t)}.
\end{equation}
\end{corollary}

\begin{proof}
  First of all, using~(\ref{eq-moment}), the monotone and dominated
  convergence theorems and~(\ref{eq-integrable}) imply that we have
\begin{equation}\label{eq-1}
\mu_k(m)=\expect_2(Z_{Y}^m)
\end{equation}
for all integers $m\geq 1$. Now a standard result of probability
theory (the ``method of moments'') states that given a positive random
variable $X$ and a sequence of positive random variables $(X_n)$, such
that $\expect(X^m)<+\infty$, $\expect(X_n^m)<+\infty$ for all $n$ and
$m$, the condition
$$
\lim_{n\ra+\infty} \expect(X_n^m)=\expect(X^m)
$$
for all $m\geq 1$ implies the convergence in law of $X_n$ to $X$,
\emph{if} the moments $\expect(X^m)$ do not grow too fast (a
sufficient, but not necessary condition). In fact, it is enough that
the power series
$$
\sum_{m\geq 0}{i^m\frac{\expect(X^m)}{m!}t^m}
$$
have a non-zero radius of convergence, which in our case holds (with
$X=Z_{Y}$) by the almost exponential upper bound for $\mu_k(m)$ in
Proposition~\ref{pr-growth}. Finally, the
formula~(\ref{eq-true-moment}) follows from~(\ref{eq-1}).
\end{proof}

\begin{example}\label{ex-zero}
As a corollary of Proposition~\ref{pr-growth} and symmetry, we have
$$
\log \mu_k(2)=2k\log\log 3k+O(k)
$$
for $k\geq 1$. 
\par
Combined with the classical lower bound for non-vanishing arising from
Cauchy's inequality, it follows that for every fixed $k\geq 1$, we
have
$$
\liminf_{h\ra +\infty}
{
\frac{1}{h^k}|\{
\uple{h}\,\mid\, |\uple{h}|\leq h\text{ and }
\sing{h}\not=0
\}|
}
\geq \frac{\mu_k(1)^2}{\mu_k(2)}\geq \exp(-(2k\log\log 3k+O(k))).
$$
\par
This is close to the truth, as one can check by noting that
we have in fact\footnote{\ This does not follow directly from
  convergence in law for $\sing{h}$, but from the absolute convergence
  and local structure of the singular series.}
$$
\lim_{h\ra +\infty}
{
\frac{1}{h^k}|\{
\uple{h}\,\mid\, |\uple{h}|\leq h\text{ and }
\sing{h}\not=0
\}|
}
=\proba_2(Z_{Y,k}\not=0)
=\prod_{p\leq k}{\proba_2(\rho_p<p)}
$$
using the almost sure absolute convergence of the random Euler product
$Z_{Y,k}$. We have the bounds
$$
\frac{(p-1)^k}{p^k}\leq \proba_2(\rho_p<p)
\leq \frac{p(p-1)^k}{p^k}
$$
(since, for $p\leq k$, a $k$-tuple will have $\rho_p<p$ only if it
omits at least one value in $\Zz/p\Zz$; the lower bound follows by
looking at those omitting $0$, for instance, and the upper one is a
union bound over the possible omitted values), from which we get
$$
-k\log\log 3k+O(k)\leq 
\log \proba_2(Z_{Y,k}\not=0)
\leq k-k\log\log 3k+O(k), 
$$
i.e., we have
$$
\proba_2(Z_{Y,k}\not=0)=\exp(-k\log\log 3k+O(k)).
$$
\par
It follows from this that if we replace the space $\Omega_1$ of all
$k$-tuples with distinct entries by the much smaller one
$$
\tilde{\Omega}_1=\{\uple{h}\in \Omega_1\,\mid\,\sing{h}\not=0\},
$$
(which still depends on $h$, with cardinality $\tilde{h}_k$), the
singular series still has a limiting distribution when interpreted as
a random variable on $\tilde{\Omega}_1$ with $h\ra +\infty$: indeed,
this is the distribution $\tilde{\nu}_k$ given by
$$
\tilde{\nu}_k(A)=\frac{\nu_k(A\cap ]0,+\infty[)}{\nu_k(]0,+\infty[)},
$$
since, for any integer $m\geq 1$, we have
$$
\frac{1}{\tilde{h}_k}\sum_{\uple{h}\in\tilde{\Omega}_1}{\sing{h}^m}=
\frac{h^*_k}{\tilde{h}_k}
\expect_1(\sing{h}^m)
\ra \frac{\mu_k(m)}{\proba_2(Z_{Y,k}\not=0)}
=\int_{[0,+\infty[}{t^md\tilde{\nu}_k(t)},
$$
as $h\ra +\infty$.
\par
Of course, those moments do not satisfy the symmetry property enjoyed
by $\mu_k(m)$.
\end{example}

\begin{remark}
  Before going on to the second part of this paper, the following
  question seems natural: are there arithmetic consequences (possibly
  conditional, similarly to Gallagher's proof
  of~(\ref{eq-gallagher-poisson})) of the existence of $m$-th moments
  of the singular series for $k$-tuples?
\end{remark}

\section{Poisson distribution for general prime patterns}
\label{sec-poisson}

In this section, we prove Theorem~\ref{pr-poisson}, essentially by
following Gallagher's reduction to averages of Euler products, which
turn out to be easily computable after application of
Proposition~\ref{pr-euler}.
\par
We fix a primitive family of polynomials $\uple{f}$ with
$\sing{f}\not=0$ (the reader may want to review the notation in the
introduction for what follows).  To apply Gallagher's method, we also
require some auxiliary families of polynomials, indexed by
$k$-tuples. Thus let $k\geq 1$ be an integer and $\uple{h}$ a
$k$-tuple of integers. For our fixed primitive $\uple{f}$, we denote
$$
\compos{f}{h}=(f_j(X+h_i))_{\stacksum{1\leq j\leq m}{1\leq i\leq k}},
$$
which is a family of $km$ integer polynomials.
\par
Technical difficulties will arise because this family may not be
primitive, even if the components of $\uple{h}$ are distinct (which is
a necessary condition), i.e., we may have an equality
$$
f_{j_1}(X+h_{i_1})=f_{j_2}(X+h_{i_2}),
$$
for some $i_1\not=i_2$, $j_1\not=j_2$. 
\par
For instance, we have $\composn{(X,X+2)}{(3,1)}=(X+3,X+1,X+5,X+3)$ (in
the case of twin primes). However, we will show that these
degeneracies have no effect for the problem at hand. Moreover,
$\compos{f}{h}$ \emph{is} primitive whenever $\uple{h}$ has distinct
arguments, in the following quite general situations:
\par
-- if $m=1$;
\par
-- if the degrees of the $f_j$ are distinct;
\par
-- if no two among the polynomials $f_j$ are related by a translation
$X\mapsto X+\alpha$, for some $\alpha\in\Zz$. 
\par
This means that the reader may well disregard the technical problems
in a first reading (for the twin primes, see also
Example~\ref{ex-poisson} which explains a special reason why the
degeneracies have no consequence then). The following lemma is already
a first step, and we will need it before proving the full statement.

\begin{lemma}\label{lm-paucity-imprim}
Let $\uple{f}$ be a primitive family and $k\geq 1$. Then for any
$h\geq 1$, we have
$$
|\{\uple{h}\,\mid\, |\uple{h}|_k\leq h,\quad
\compos{f}{h}\text{ is not primitive}\}|\ll h^{k-1}
$$
where the implied constant depends only on $k$ and $m$.
\end{lemma}

\begin{proof}
  Let $I$ be the set of $k$-tuples $\uple{h}$ with distinct components
  such that $\compos{f}{h}$ is not primitive.  If $\uple{h}\in I$,
  then there exists at least one relation of the type
\begin{equation}\label{eq-rel1}
f_{j_1}(X+h_{i_1})=f_{j_2}(X+h_{i_2}),\quad\quad
i_1\not=i_2,\quad j_1\not=j_2,
\end{equation}
hence
$$
f_{j_1}(X)=f_{j_2}(X+h_{i_2}-h_{i_1}),
$$
so the two polynomials differ by a ``shift''. Let $\mathcal{R}$ be the
set of pairs $(j_1,j_2)$ for which
$$
f_{j_1}(X)=f_{j_2}(X+\delta(j_1,j_2))
$$
for some integer $\delta(j_1,j_2)\not=0$. Because the polynomials
involved are non-constant, this integer is indeed unique.  The
cardinality of $\mathcal{R}$ is bounded in terms of $m$ only, and from
the above, any $k$-tuple $\uple{h}\in I$ must satisfy at least one
relation
$$
h_{i_1}-h_{i_2}=\delta(j_1,j_2),
$$
for some $i_1\not=i_2$ and $(j_1,j_2)\in\mathcal{R}$. Each such
relation is valid for at most $h^{k-1}$ among the $k$-tuples with
$|\uple{h}|\leq h$. 
\end{proof}

We will deduce Theorem~\ref{pr-poisson} from the following
(unconditional) result, which is another instance of average of Euler
products:

\begin{proposition}\label{pr-identity}
Let $\uple{f}=(f_1,\ldots, f_m)$ be a primitive family and $k\geq 1$
an integer.  Then we have
$$
\lim_{h\ra +\infty}{
\frac{1}{h^k}{
\sums_{|\uple{h}|\leq h}
\singn{\compos{f}{h}}
}
}=\sing{f}^k,
$$
where $\sums$ here restricts the summation to those $k$-tuples for
which $\compos{f}{h}$ is primitive.
\end{proposition}

\begin{remark}
Taking $\uple{f}=(X)$, with $\sing{f}=1$ and
$\compos{f}{h}=(X+h_1,\ldots, X+h_k)$, we recover once
more Gallagher's result~(\ref{eq-gallagher}).
\end{remark}

We have the following complementary statement, which is also
unconditional (recall that, in many cases, it holds for trivial
reasons; it does \emph{not} follow trivially from
Lemma~\ref{lm-paucity-imprim} because although fewer $k$-tuples are
concerned, the number of prime seeds increases when $\compos{f}{h}$ is
not primitive).

\begin{lemma}
\label{lm-imprim}
Let $\uple{f}=(f_1,\ldots, f_m)$ be a primitive family with
$\sing{f}\not=0$, and $k\geq 1$ an integer.  Then for any $N\geq 2$,
if $h\leq \lambda (\log N)^m$ for some $\lambda>0$, and for any
$\eps>0$, we have
$$
\sums_{\stacksum{|\uple{h}|_k\leq h}
{\compos{f}{h}\text{ not primitive}}} 
{
\pi(N;\compos{f}{h})
}\ll \frac{N}{(\log N)^{1-\eps}}
$$
where $\sums$ restricts the sum to those $k$-tuples with distinct
entries, and where the implied constant depends only on $k$,
$\uple{f}$, $\lambda$ and $\eps$.
\end{lemma}

Here is the proof of the (conditional) Poisson distribution, assuming
those two results.

\begin{proof}[Proof of Theorem~\ref{pr-poisson}]
The argument is essentially identical with that of Gallagher, but we
reproduce it for completeness, and so that the necessary uniformity in
the Bateman-Horn conjecture becomes clear.
\par
Because the Poisson distribution is characterized by its moments, it
is enough to prove that for any fixed integer $k\geq 1$, we have
$$
\frac{1}{N}\sum_{n\leq N}{
\Bigl(
\pi(n+\lambda\delta(N,\uple{f});\uple{f})-\pi(n;\uple{f})
\Bigr)^k
}\ra \expect(P_{\lambda}^k),\quad\quad
\text{as }N\ra+\infty,
$$
where $P_{\lambda}$ is any Poisson random variable with mean
$\lambda$. 
\par
Write $h=\lambda\delta(N,\uple{f})$. Expanding the left-hand side, we
obtain
$$
\frac{1}{N}\sum_{n\leq N}
{
\Bigl(
\multsum_{\stacksum{n<m_i\leq n+h}{m_i\text{ $\uple{f}$-prime seed}}}
{1}
\Bigr)
}
$$
where there are $k$ sums over $m_1$, \ldots, $m_k$. Write $m_i=n+h_i$,
so that $1\leq h_i\leq h$, and the condition becomes that $f_j(n+h_i)$
is prime for all $i$ and $j$, i.e., that $n$ be an
$\compos{f}{h}$-prime seed. Exchanging the order of
summation, we get
$$
\frac{1}{N}\sum_{|\uple{h}|_k\leq h}
{\pi(N;\compos{f}{h})}.
$$
\par
Before applying~(\ref{eq-bh-unif}), we need to account for the
$k$-uples which do not necessarily have distinct components, and for
those where $\compos{f}{h}$ is \emph{not} primitive.
\par
For this, observe first that $\pi(N;\compos{f}{h})$ only depends on
the set containing the components of the $k$-tuple $\uple{h}$. This
justifies the fact that the reorderings that follow are permissible.
For each $r$, $1\leq r\leq k$, and each $r$-tuple $\uple{h}'$ with
distinct components, the set of those $k$-tuples for which the set of
values is given by the set of components of $\uple{h}'$ has
cardinality depending only on $r$ and $k$, but independent of
$\uple{h}'$, and in fact it is given by $\stirling{k}{r}$ (one can
assume that $\uple{h}'=(1,\ldots, r)$, and obtain a bijection
$$
\left\{
\begin{array}{ccl}
\{\text{suitable $k$-tuples}\}&\ra& 
\{\text{surjective maps } \{1,\ldots, k\}\ra
\{1,\ldots, r\}\}\\
\uple{h}&\mapsto& (f\,:\, i\mapsto h_i)
\end{array}
\right.
$$
between the two sets).
\par
Then we can write
$$
\frac{1}{N}\sum_{|\uple{h}|_k\leq h}
{\pi(N;\compos{f}{h})}=
\frac{1}{N}\sum_{r=1}^k{
\frac{1}{r!}\stirling{k}{r}
\sums_{|\uple{h}'|_r\leq h}{
\pi(N;\compos{f}{h'})
}}
$$
where we divide by $r!$ because we sum over all $r$-tuples instead of
only ordered ones, and $\sums$ restricts to $r$-tuples with distinct
entries.
\par
Now, for each $r$, we separate the sum over $r$-tuples for which
$\compos{f}{h'}$ is primitive from the other
subsum. Applying~(\ref{eq-bh-unif}) and using the easy bound
$$
c(\compos{f}{h'})\ll c(\uple{f})|\uple{h'}|_r^{\max \deg(f_j)},
$$
(where the implied constant depends on $r$ and $\uple{f}$) the first
sum (still denoted $\sums$) is equal to
$$
\sum_{r=1}^k{ \frac{1}{r!}\stirling{k}{r}
  \frac{1}{\peg(\uple{f})^r}\frac{1}{(\log N)^{rm}}
  \sums_{|\uple{h}'|_r\leq h}{ \singn{\compos{f}{h'}}
\Bigl(1+O\Bigl(
\frac{h^{\eps}}{\log N}
\Bigr)\Bigr)
 } },
$$
for any $\eps>0$, where the implied constant depends on $\uple{f}$,
$k$ and $\eps$. Using Proposition~\ref{pr-identity} and the choice of
$h=\lambda \peg(\uple{f})\sing{f}^{-1}(\log N)^m$, this converges as
$N\ra +\infty$ to the limit
$$
\sum_{r=1}^k{ \frac{\lambda^r}{r!}\stirling{k}{r}},
$$
which is well-known to be the $k$-th moment of a Poisson distribution
with mean $\lambda$ (this is checked by Gallagher for instance,
see~\cite[\S 3]{gallagher}). Hence, to conclude the proof, we need
only notice that Lemma~\ref{lm-imprim} (applied with $k=r$ for $1\leq
r\leq k$) implies (taking $\eps=1/2$ for concreteness) that the
complementary sum is bounded by
$$
\frac{1}{N}\sum_{r=1}^k{ \frac{1}{r!}\stirling{k}{r}
  \sum_{\stacksum{|\uple{h}'|_r\leq h} {\compos{f}{h'}\text{ not
        primitive}}}{ \pi(N;\compos{f}{h'}) }}\ll (\log N)^{-1/2}
$$
for $N\geq 2$, where the implied constant depends on $k$, $\uple{f}$
and $\lambda$. Hence this second contribution goes to $0$ as $N\ra
+\infty$, as desired.
\end{proof}

We now prove Proposition~\ref{pr-identity}. This is the conjunction of
the two following lemmas, where we use the same notation as in
Section~\ref{sec-moments-singular}, but change a bit the definition of
probability spaces. Precisely,
$$
\Omega_2=\prod_p{(\Zz/p\Zz)^k}
$$
is unchanged, but we let
$$
\Omega_1=\{\uple{h}=(h_1,\ldots, h_k)\,\mid\, 
1\leq h_i\leq h,\quad \compos{f}{h}\text{ is primitive}
\}
$$
with the counting probability measure (note that the condition forces
$\uple{h}$ to have distinct coordinates). By
Lemma~\ref{lm-paucity-imprim}, note that we have
\begin{equation}\label{eq-cons-imprim}
|\Omega_1|\sim h^{k}\quad\text{ as } h\ra +\infty.
\end{equation}
\par
The next lemma shows that the average of Euler product involved can be
computed as if the components where independent:

\begin{lemma}\label{lm-apply}
  Let $\sing{f}=(f_1,\ldots, f_m)$ be a primitive family with
  $\sing{f}\not=0$. Then for any $k\geq 1$, we have
\begin{align*}
\lim_{h\ra +\infty}{
\frac{1}{h^k}{
\sum_{|\uple{h}|\leq h}
\singn{\compos{f}{h}}
}
}&
=
\lim_{h\ra +\infty}{
\expect_1\Bigl(\prod_p{
\Bigl(1-\frac{1}{p}\Bigr)^{-km}
\Bigl(1-
\frac{\nu_{p,\uple{f}}}{p}
\Bigr)
}\Bigr)}\\
&=\prod_{p}{\expect_2\Bigl(
\Bigl(1-\frac{1}{p}\Bigr)^{-km}
\Bigl(1-\frac{\rho_{p,\uple{f}}}{p}\Bigr)\Bigr)},
\end{align*}
where
\begin{gather*}
\nu_{p,\uple{f}}(\uple{h})=\nu_p(\compos{f}{h})
\text{ for } \uple{h}=(h_1,\ldots, h_k)\text{ with }
h_i\geq 1,
\\
\rho_{p,\uple{f}}(\uple{h})=
|\{
x\in \Zz/p\Zz\,\mid\, f_j(x+h_i)=0\text{ for some $i$, $j$}
\}|
\text{ for } \uple{h}\in (\Zz/p\Zz)^r.
\end{gather*}
\end{lemma}

The second lemma computes the limit locally:

\begin{lemma}\label{lm-end}
Let $\uple{f}=(f_1,\ldots, f_m)$ be a primitive family. Then for any
$k\geq 1$ and any prime $p$, we have
$$
\expect_2\Bigl(
\Bigl(1-\frac{1}{p}\Bigr)^{-km}
\Bigl(1-\frac{\rho_{p,\uple{f}}}{p}\Bigr)\Bigr)=
\Bigl(1-\frac{1}{p}\Bigr)^{-km}
\Bigl(1-\frac{\nu_p(\uple{f})}{p}\Bigr)^k.
$$
\end{lemma}

Looking at the definition~(\ref{eq-singf}) of $\sing{f}$, both lemmas
together prove Proposition~\ref{pr-identity}.  We start by proving
Lemma~\ref{lm-end} because Lemma~\ref{lm-apply} is certainly plausible
enough in view of Section~\ref{sec-moments-singular}, and the reader
may be more interested by the final formal flourish.

\begin{proof}[Proof of Lemma~\ref{lm-end}]
It suffices to compute
$$
\expect_2\Bigl(
1-\frac{\rho_{p,\uple{f}}}{p}\Bigr)
$$
since the other factor is the same on both sides. We argue
probabilistically, although one can also just expand the various sums
(and do the same steps in a different language, as we did when proving
the symmetry~(\ref{eq-symmetric})). We can write
$$
1-\frac{\rho_{p,\uple{f}}}{p}=\frac{1}{p}|\Zz/p\Zz-M|
$$
where $M\subset \Zz/p\Zz$ is the (random) subset of those
$x\in\Zz/p\Zz$ such that $f_j(x+h_i)=0$ for some $i$ and $j$. We write
$$
|\Zz/p\Zz-M|=\sum_{x\in\Zz/p\Zz}{(1-\chi_M(x))}
$$
where $\chi_M(x)$ is the random variable equal to one if $x\in M$ and
zero otherwise. We have
$$
1-\chi_M(x)=\prod_{1\leq i\leq k}\prod_{1\leq j\leq m}{
(1-\charfun_{\{f_j(x+h_i)=0\}})
}=\prod_{1\leq i\leq k}{\xi_{\uple{f},i}(x)},
$$
say. Since $\xi_{\uple{f},i}(x)$ only involves the $i$-th component of the
random $\uple{h}\in\Omega_2$, the family $(\xi_{\uple{f},i}(x))$ is an
independent $k$-tuple of random variables. Consequently we derive
\begin{align*}
\expect_2\Bigl(
1-\frac{\rho_{p,\uple{f}}}{p}\Bigr)&=
\frac{1}{p}\sum_{x\in \Zz/p\Zz}{\expect_2\Bigl(\prod_{1\leq i\leq k}
{\xi_{\uple{f},i}(x)}\Bigr)}\\
&=\frac{1}{p}\sum_{x\in\Zz/p\Zz}
{\prod_{1\leq i\leq k}{\expect_2(\xi_{\uple{f},i}(x))}}.
\end{align*}
\par
To conclude we notice that for every $x$ and $i$, $\uple{h}\mapsto
x+h_i$ is identically (uniformly) distributed, so that all
$\xi_{\uple{f},i}(x)$ are identically distributed like
$$
\xi_{\uple{f}}=\xi_{\uple{f},1}(0)
=\prod_{1\leq j\leq m}{(1-\charfun_{\{f_j(h_1)=0\}})}.
$$
\par
Hence all $x$ give the same contribution, and we derive that
$$
\expect_2\Bigl(
1-\frac{\rho_{p,\uple{f}}}{p}\Bigr)
=\expect_2(\xi_{\uple{f}})^k
=\proba_2(f_1(h_1)\cdots f_m(h_1)\not=0)^k=
\Bigl(1-\frac{\nu_{p}(\uple{f})}{p}\Bigr)^k,
$$
since $h_1$ is uniformly distributed in $\Zz/p\Zz$.
\end{proof}

To prove Lemma~\ref{lm-apply}, we wish to apply
Proposition~\ref{pr-euler}. A complication is that, if
$\peg(\uple{f})\not=1$, the singular series $\singn{\compos{f}{h}}$
are not defined by absolutely convergent products, and therefore the
result is not directly applicable. However, we can bypass this
difficulty here without significant work because of the following
fact: all the relevant Euler products can be \emph{uniformly}
``renormalized'' to absolutely convergent ones. This is the content of
the next lemma.

\begin{lemma}
  Let $\uple{f}$ be a primitive family with $\sing{f}\not=0$, and let
  $k\geq 1$ be an integer. There exist real numbers
  $\gamma_p(\uple{f})> 0$, for all primes $p$, such that the product
$$
\prod_p{\gamma_p(\uple{f})}
$$ 
converges, and such that the following hold:
\par
\emph{(1)} For all prime $p$, and all $k$-tuple $\uple{h}\in
(\Zz/p\Zz)^k$, we have
$$
\Bigl(1-\frac{1}{p}\Bigr)^{-km}
\Bigl(1-\frac{\rho_{p,\uple{f}}(\uple{h})}{p}\Bigr)
=\gamma_p(\uple{f})
\times (1+X_{p,\uple{f}}(\uple{h}))
$$
for some coefficients $X_{p,\uple{f}}(\uple{h})$, and for all
$k$-tuple of integers $\uple{h}$ such that $\compos{f}{h}$ is
primitive, the product
\begin{equation}\label{eq-renorm}
\prod_{p}{(1+X_{p,\uple{f}}(\uple{h}))}
\end{equation}
is absolutely convergent.
\par
\emph{(2)} We have
$$
\lim_{h\ra +\infty}{\frac{1}{h^k}
\sums_{|\uple{h}|_k\leq h}{
\prod_p{(1+X_{p,\uple{f}}(\uple{h}))}
}}
=\prod_{p}{(1+\expect_2(
X_{p,\uple{f}}))},
$$
where the sum is over $k$-tuples with $\compos{f}{h}$ primitive.
\end{lemma}

\begin{proof}
  (1) To define $\gamma_p(\uple{f})$, let $\theta_j$, $1\leq j\leq m$,
  be a complex root of the irreducible polynomial $f_j$, and let
  $K_j=\Qq(\theta_j)$ be the extension of $\Qq$ of degree $\deg(f_j)$
  generated by $\theta_j$. Then put
$$
\gamma_p(\uple{f})=\prod_{1\leq i\leq m}{
\Bigl(1-\frac{1}{p}\Bigr)^{k(r_{j}(p)-1)}}
$$
where $r_{j}(n)$, for $n\geq 1$, is the number of prime ideals of norm
$n$ in the ring of integers of $K_j$. In view of this definition, to
check first that the product of $\gamma_p(\uple{f})$ converges, we can
do so for each $f_j$ separately. Then the statement follows, after
taking the logarithm of a partial product over $p\leq X$, from the
well-known asymptotic formula
$$
\sum_{p\leq X}{\frac{r_{j}(p)}{p}}=\sum_{p\leq X}{\frac{1}{p}}+c(K_j)
+O((\log X)^{-1})
$$
for $X\geq 2$, where $c(K_j)$ is a constant depending only on $K_j$,
and the implied constant also depends only on $K_j$.
\par
It therefore remains to prove that the product~(\ref{eq-renorm}) is
absolutely convergent for any $k$-tuple of integers $\uple{h}$ with
$\compos{f}{h}$ primitive. To do so, we claim that there exists an
integer $D(\uple{h})\geq 1$ (which may also depend on $\uple{f}$) such
that, for $p\nmid D(\uple{h})$, we have
\begin{equation}\label{eq-distinct-zeros}
\rho_{p,\uple{f}}(\uple{h})
=k\sum_{j=1}^m{\nu_p(f_j)}=k\sum_{j=1}^m{r_{j}(p)}.
\end{equation}
\par
The desired convergence then follows from that of
$$
\prod_{p\nmid D(\uple{h})}{\gamma_p(\uple{f})^{-1}
\Bigl(1-\frac{1}{p}\Bigr)^{-km}
\Bigl(1-\frac{\rho_{p,\uple{f}}(\uple{h})}{p}\Bigr)}
=\prod_{p\nmid D(\uple{h})}{
\Bigl(
1-\frac{1}{p}
\Bigr)^{-\rho_{p,\uple{f}}(\uple{h})}
\Bigl(1-\frac{\rho_{p,\uple{f}}(\uple{h}}{p}\Bigr)
},
$$
and the latter is clear since the $p$-factor can be written
$1+O(p^{-2})$, where the implied constant depends only on $k$ and
$\uple{f}$.
\par
The existence of $D(\uple{h})$ is easy; first, let
$$
D_1(\uple{h})=\Bigl| \prod_{(i,j)\not=(i',j')}{
  \res(f_j(X+h_i),f_{j'}(X+h_{i'})) }\Bigr| ,
$$
where $\res(\cdot,\cdot)$ is the resultant of two polynomials. By
compatibility of the resultant with reduction modulo $p$, we have
$p\mid D_1(\uple{h})$ if and only if, for some $(i,j)\not=(i',j')$,
there exists a common zero $x\in\Zz/p\Zz$ of $f_{j}(X+h_i)$ and
$f_{j'}(X+h_{i'})$. By contraposition, we first obtain
$$
\rho_{p,\uple{f}}(\uple{h})=
k\nu_{p}(\uple{f})=
k\sum_{j=1}^m{\nu_p(f_j)},
$$
for $p\nmid D_1(\uple{h})$ (the sets of zeros modulo $p$ of the
components of $\compos{f}{h}$ are then distinct, and obviously there
are as many, namely the sum $\nu_p(\uple{f})$ of the $\nu_p(f_j)$, for
each of the $k$ shifts $h_i$).
\par
Next, it is a standard fact of algebraic number theory that for each
$j$, there exists an integer $\Delta_j\geq 1$ such that
$\nu_p(f_j)=r_{j}(p)$ for $p\nmid \Delta_j$. Thus we can take
$$
D(\uple{h})=D_1(\uple{h})\prod_{1\leq j\leq m}{\Delta_j}
$$
to obtain the second equality in~(\ref{eq-distinct-zeros}).
\par
Note that $D(\uple{h})$ is non-zero (hence $\geq 1$) because
otherwise, there would exist a common zero $\theta\in\Cc$ of
$f_{j}(X+h_i)$ and $f_{j'}(X+h_{i'})$, and because those are
irreducible integral primitive\footnote{\ In the sense that the gcd of
  their coefficients is $1$.} polynomials with positive leading
coefficient, this is only possible if
$$
f_{j}(X+h_i)=f_{j'}(X+h_{i'}),
$$
which is excluded by the assumption that $\compos{f}{h}$ be primitive.
\par
Note in passing the estimate
$$
D(\uple{h})\ll (2|\uple{h}|_k)^{2k^2m\sum \deg(f_j)}
$$
for all $\uple{h}$, where the implied constant depends only on
$\uple{f}$; this follows straightforwardly from the determinant
expression of the resultant in $D_1(\uple{h})$ (see, e.g.,~\cite[\S
V.10]{lang}).
\par
(2) With the bounds we have proved on $X_{p,\uple{f}}(\uple{h})$
(leading to an analogue of Lemma~\ref{lm-apnu}), and the estimate on
$D(\uple{h})$ (analogue of~(\ref{eq-discrim})), together with
Lemma~\ref{lm-paucity-imprim} to ensure that the equidistribution of
$k$-tuples modulo squarefree integers $q$ remains valid (compare
with~(\ref{eq-mid-equidist})), we can pretty much follow the steps of
the proof of Theorem~\ref{pr-gallagher-moments}. We also
use~(\ref{eq-cons-imprim}) to go from the limit of the expectation on
$\Omega_1$ to summing over $k$-tuples normalized by $1/h^k$ and taking
$h\ra +\infty$. The details are left to the reader.
\end{proof}

\begin{proof}[Proof of Lemma~\ref{lm-apply}]
We have first
\begin{align*}
\frac{1}{h^k}{
\sum_{|\uple{h}|\leq h}
\singn{\compos{f}{h}}
}
&=
\frac{1}{h^k}{
\sum_{|\uple{h}|\leq h}
\Bigl(\prod_p{\gamma_p(\uple{f})}\Bigr)
\prod_{p}{(1+X_{p,\uple{f}}(\uple{h}))}
}
\\&\ra 
\Bigl(\prod_p{\gamma_p(\uple{f})}\Bigr)
 \prod_{p}{(1+\expect_2(
X_{p,\uple{f}}))}\quad\text{ as } h\ra +\infty,
\end{align*}
by the above, and then we can simply write this limit as
\begin{align*}
  \Bigl(\prod_p{\gamma_p(\uple{f})}\Bigr) \prod_{p}{(1+\expect_2(
    X_{p,\uple{f}}))} &=\prod_{p}{ \expect_2(
    \gamma_p(\uple{f})(1+X_{p,\uple{f}}) )
  }\\
  &=\prod_p{\expect_2\Bigl( \Bigl(1-\frac{1}{p}\Bigr)^{-km}
    \Bigl(1-\frac{\rho_{p,\uple{f}}(\uple{h})}{p}\Bigr) \Bigr)}.
\end{align*}
\end{proof}

We conclude with the last remaining part of the proof, namely
Lemma~\ref{lm-imprim}.  The following proof can almost certainly be
improved, but although the statement becomes fairly clear after
checking one or two examples, the author has not found a cleaner way
to deal with the apparent possibilities of combinatorial
complications. The point is that as $\compos{f}{h}$ becomes ``less
primitive'' (i.e., there are less distinct elements among the $km$
polynomials involved), the number of prime seeds $\leq N$ should
increase (by a power of $(\log N)$), but also the number of $k$-tuples
with this property diminishes (by a power of $h\leq \lambda (\log
N)^m$), and this gain has to compensate for the loss.

\begin{proof}[Proof of Lemma~\ref{lm-imprim}]
We first quote a standard sieve upper-bound for an individual
primitive family $\uple{f}$ (with $m$ elements), which is uniform, and
which allows us to prove the lemma unconditionally: for $N\geq 2$, for
any $k$-tuple $\uple{h}$ with distinct elements for which
$\compos{f}{h}$ contains $\ell$ distinct components, we have
\begin{equation}\label{eq-sieve}
\pi(N;\compos{f}{h})\ll (\log\log 3|\uple{h}|)^{km}
\frac{N}{(\log N)^{\ell}},
\end{equation}
where the implied constant depends only on $k$ and
$\uple{f}$. Precisely,~(\ref{eq-sieve}) for $k$-tuples follows
immediately from, e.g, Th. 2.3 in~\cite{halberstam-richert}, and it is
easy to adapt this to the case at hand since uniformity is only asked
with respect to $\uple{h}$.
Note also that, since the application we give is conditional on much
stronger statements like~(\ref{eq-bh-unif}), we could also apply the
latter for this purpose.
\par
Now, as in the proof of Lemma~\ref{lm-paucity-imprim}, we denote by
$I$ the set of $k$-tuples $\uple{h}$ with distinct components such
that $\compos{f}{h}$ is not primitive. Recall $\mathcal{R}$ is the set
of pairs $(j_1,j_2)$ for which
$$
f_{j_1}(X)=f_{j_2}(X+\delta(j_1,j_2))
$$
for some (unique) integer $\delta(j_1,j_2)\not=0$.
\par
We continue as follows: for an $\uple{h}\in I$, let
$\Gamma_{\uple{h}}$ be the graph with vertex set $\{1,\ldots, k\}$ and
with (unoriented) edges $(i_1,i_2)$ corresponding to those indices for
which the relation
\begin{equation}\label{eq-relation}
h_{i_1}-h_{i_2}=\delta(j_1,j_2)
\end{equation}
holds for some $(j_1,j_2)\in\mathcal{R}$; the proof of
Lemma~\ref{lm-paucity-imprim} shows that there is at least one
edge. Because the number of possibilities for $\Gamma_{\uple{h}}$ is
clearly bounded in terms of $k$ only, and we allow a constant
depending on $k$ in our estimate, we may continue by fixing one
possible graph $\Gamma$ and assuming that \emph{all} $\uple{h}\in I$
satisfy $\Gamma_{\uple{h}}=\Gamma$.
\par
This being done, we first estimate from above the number of $k$-tuples
which lie in $I$ (under the above assumption that the graph is
fixed!). We claim that
\begin{equation}\label{eq-upper}
|\{\uple{h}\in I\,\mid\, |\uple{h}|\leq h\}|\leq h^{c}
\end{equation}
where $c=|\pi_0(\Gamma)|$ is the number of connected components of
$\Gamma$. 
\par
To see this, notice that each connected component $C$ corresponds to a
set of variables which are \emph{independent} of all others, so that
$I$ is the product over the connected components of sets $I_C$ of
$|C|$-tuples satisfying the relations~(\ref{eq-relation}) dictated by
$C$. Now we have
$$
|\{\uple{h}\in I_C\,\mid\, |\uple{h}|\leq h\}|\leq h,
$$
because $C$ is connected: if we fix some vertex $i_0$ of $C$, then for
any choice of $h_{i_0}$, the value of $h_i$ is determined by means of
the relations~(\ref{eq-relation}) for all vertices $i$ of $C$, using
induction on the length of a path from $i_0$ to $i$ (which exists by
connectedness).
\par
Taking the product over $C$ of these individual upper bounds, we
obtain the desired estimate~(\ref{eq-upper}).
\par
We next need to estimate from below the number of distinct elements in
the family $\compos{f}{h}$ for a fixed $\uple{h}\in I$ (still under
the assumption that the graph $\Gamma_{\uple{h}}=\Gamma$ is
fixed). 
\par
Let again $C$ be a connected component of the graph $\Gamma$. We
consider the set (say $\{\compos{f}{h}\}_C$) of polynomials of the
form $f_j(X+h_i)$, where $1\leq j\leq m$ and $i$ is a vertex of
$C$. We claim this set contains at least $m+1$ distinct polynomials if
$C$ has at least $2$ vertices, and $m$ if $C$ is a singleton. Indeed,
fixing a vertex $i_0$ of $C$, the set contains the polynomials
$f_j(X+h_{i_0})$, which are distinct since $\uple{f}$ is a primitive
family. This already takes care of the case where $C$ is a singleton,
so assume now that $C$ contains at least another vertex $i$. If all
the $m$ distinct polynomials $f_j(X+h_i)$ were already in the set
$\{f_j(X+h_{i_0})\}$, this would define a permutation $\sigma$ of
$\{1,\ldots, m\}$ such that
$$
f_j(X+h_i)=f_{\sigma(j)}(X+h_{i_0}),\quad 1\leq j\leq m.
$$
\par
Consider a cycle $(j_1,\ldots, j_{\ell})$ of length $\ell$ in the
decomposition of $\sigma$; applying the identity to $j_1$,
$\sigma(j_1)=j_2$, etc, in turn, we derive the identity
$$
f_{j_1}(X)=f_{\sigma^{\ell}(j_1)}(X+(\ell-1)(h_{i_0}-h_i))=
f_{j_1}(X+(\ell-1)(h_{i_0}-h_i)).
$$
\par
Since $f_j$ is non-constant and $h_{i_0}\not= h_i$, we deduce that
$\ell=1$; this holding for all cycles in $\sigma$ would mean that
$\sigma$ is the identity, but then $f_1(X+h_i)=f_1(X+h_{i_0})$ again
contradicts the fact that $\uple{h}$ has distinct components. This
means that $\sigma$ can not exist, and so the set $\{f_j(X+h_{i})\}$
contains at least one polynomial not among the first $m$ ones, which
was our objective.
\par
Next observe that, by the very definition of the graph $\Gamma$, the
sets $\{\compos{f}{h}\}_C$ are disjoint when $C$ runs over the
connected components of $\Gamma$, and hence we find that any
$\compos{f}{h}$ contains at least $cm+d$ elements, where $d$ is the
number of connected components of $\Gamma$ which are not
singletons. Note that $d\geq 1$, because $\Gamma$ has at least one
edge. 
\par
We finally estimate the contribution of $k$-tuples in $I$
using~(\ref{eq-sieve}) and~(\ref{eq-upper}): we obtain
$$
\frac{1}{N}\sum_{\stacksum{\uple{h}\in I}{|\uple{h}|\leq h}}{
\pi(N;\compos{f}{h})
}
\ll h^{c}(\log 2h)^{km}(\log N)^{-cm-d}
$$
where the implied constant depends on $k$ and $\uple{f}$. If $h\leq
\lambda (\log N)^m$, as assumed in Lemma~\ref{lm-imprim}, we obtain
$$
\frac{1}{N}\sum_{\stacksum{\uple{h}\in I}{|\uple{h}|\leq h}}{
\pi(N;\compos{f}{h})
}
\ll (\log N)^{-d+\eps}
$$
for any $\eps>0$, where the implied constant depends on $k$,
$\lambda$, $\uple{f}$ and $\eps$. Since $d\geq 1$, the lemma is
finally proved.
\end{proof}

\begin{remark}
  The gain of $(\log N)^{-1}$ is indeed the best possible in
  general. Consider for example the primitive family
  $\uple{f}=(f_1,f_2,f_3)=(X^2+7,(X+2)^2+7,(X+4)^2+7)$ for which it is
  easy to check that $\sing{f}\not=0$ ($7$ is not a square modulo $3$
  or $5$, and each $f_j(0)$ is odd). We have relations
  $f_1(X+2)=f_2(X)$, $f_2(X+2)=f_3(X)$.
\par
Consider $k=2$. If we look at $2$-tuples $\uple{h}=(h_1,h_2)$ for
which $h_2=h_1+2$, we obtain
\begin{gather*}
  \compos{f}{h}=(f_1(X+h_1),f_2(X+h_1),f_3(X+h_1),\\
 \hspace{3cm} f_1(X+h_2),f_2(X+h_2),f_3(X+h_2))\\
  =
  (f_1(X+h_1),f_1(X+h_1+2),f_1(X+h_1+4),\\
  \hspace{3cm}f_1(X+h_2),f_1(X+h_2+2),f_1(X+h_2+4))\\
  =
  (f_1(X+h_2-2),f_1(X+h_2),f_1(X+h_2+2),\\
  \hspace{3cm}f_1(X+h_2),f_1(X+h_2+2),f_1(X+h_2+4)),
\end{gather*}
which contains $4$ distinct polynomials. With $h\asymp \lambda(\log
N)^3$, those $2$-tuples with $|\uple{h}|\leq h$ contribute about
$N(\log N)^{3-4}$ to the sum of Lemma~\ref{lm-imprim}
(under~(\ref{eq-bh-unif}), of course).
\end{remark}

Finally, here are a few examples.

\begin{example}\label{ex-poisson}
  (1) If we take $\uple{f}_1=(X,X+2)$, we obtain that the number of
  twin primes $(p, p+2)$ with $n<p\leq n+\lambda (\log n)^2$ should be
  approximately distributed like a Poisson random variable with mean
$$
2\lambda \prod_{p\geq 3}{\Bigl(1-\frac{1}{(p-1)^2}\Bigr)}\approx
1.320336593\ldots \lambda.
$$
\par
Similarly, if we take $\uple{f}_2=(X,2X+1)$, we find that the number
of Germain primes (i.e., primes $p$ with $2p+1$ also prime) with
$n<p\leq n+\lambda(\log n)^2$ should be approximately distributed like
a Poisson random variable with mean
$$
\lambda \singn{\uple{f}_2}=2\lambda \prod_{p\geq
  3}{\Bigl(1-\frac{1}{(p-1)^2}\Bigr)}.
$$
\par
Two further remarks are interesting here. First, the proof of
Theorem~\ref{pr-poisson} shows that whenever $\uple{f}$ consists
of linear polynomials (in particuler for those two results), ``only''
the (uniform) Hardy-Littlewood conjecture is needed. In other words,
no assumption is required beyond those of Gallagher's original result
for the primes themselves.
\par
Secondly, if one is interested in the case of twin primes in
particular, Lemma~\ref{lm-imprim} has a trivial proof from the
following coincidence: if $\uple{f}=(X,X+2)$, $\uple{h}$ has distinct
entries, and $\compos{f}{h}$ is not primitive, then
$$
\singn{\compos{f}{h}}=0,\quad\quad
\pi(N;\compos{f}{h})\leq 1.
$$
\par
Indeed, if $\compos{f}{h}$ is not primitive, we have $k\geq 2$ and an
equality $h_{i_2}=h_{i_1}+2$ for some $i_1$, $i_2$. The family
$\compos{f}{h}$ contains in particular the three polynomials
$X+h_{i_1}$, $X+h_{i_2}=X+h_{i_1}+2$ and
$X+h_{i_2}+2=X+h_{i_1}+4$. Hence, to be a prime seed for
$\compos{f}{h}$, an integer $n\geq 1$ must be such that, in
particular, the triple $(n+h_{i_1},n+h_{i_1}+2,n+h_{i_1}+4)$ consists
of prime numbers. But those three numbers are distinct modulo $3$,
showing that $\nu_3(\compos{f}{h})=3$, and the only possible case is
$(n,n+2,n+4)=(3,5,7)$. (Examples such as $\uple{f}=(X^2+7,(X+2)^2+7)$
and $\uple{h}=(3,1)$ show that this special situation where
imprimitive $k$-tuples lead to vanishing singular series for
$\compos{f}{h}$ is indeed a coincidence).
\par
(2) If we take $\uple{f}_3=(X^2+1)$, and renormalize in an obvious way,
we find that the number of primes of the form $p=n^2+1$ in an interval
of the form $N^2<n\leq (N+\lambda (\log N))^2$ should be
approximately distributed like a Poisson random variable with mean
$$
\lambda \singn{\uple{f}_3}=\frac{4 \lambda }{\pi}
\prod_{p\equiv 1\mods{4}}{\Bigl(1-\frac{1}{(p-1)^2}\Bigr)}
\prod_{p\equiv 3\mods{4}}{\Bigl(1-\frac{1}{p^2-1}\Bigr)}.
$$
\end{example}

\end{document}